\documentclass[11pt, reqno]{amsart}
\usepackage{indentfirst, amssymb, amsmath, amsthm, mathrsfs, setspace, indentfirst, enumerate,  mathrsfs, amsmath, amsthm, graphicx, MnSymbol}
\usepackage[bookmarksnumbered, colorlinks, plainpages]{hyperref}
\usepackage{mathrsfs}
\usepackage{tikz}
\usepackage{float}
\usepackage{tikz}
\usetikzlibrary{positioning,arrows.meta,shapes.geometric,calc}
\usepackage{pgfplots}
\usepackage{tabularx, colortbl, xcolor}
\usepackage{array, enumitem}
\pgfplotsset{compat=1.18}
\usetikzlibrary{3d}
\usepackage{tikz-3dplot}
\usepackage{float}
\usepackage{placeins}
\usepackage[most]{tcolorbox}
 \usepackage{xcolor}
\tdplotsetmaincoords{70}{120}
\newtheorem*{cor A}{Corollary A}
\newtheorem*{cor B}{Corollary B}

\newtheorem{theo}{Theorem}[section]
\newtheorem{lem}{Lemma}[section]

\newtheorem{defi}{Definition}[section]

\newcommand{\ol}{\overline}
\newcommand{\be}{\begin{equation}}
	\newcommand{\ee}{\end{equation}}
\newcommand{\beas}{\begin{eqnarray*}}
	\newcommand{\eeas}{\end{eqnarray*}}
\newcommand{\bea}{\begin{eqnarray}}
	\newcommand{\eea}{\end{eqnarray}}

\numberwithin{equation}{section}
\begin{document}
	\title[C\MakeLowercase{oefficients estimates and \MakeUppercase{H}ankel determinant......}]{\LARGE C\LARGE\MakeLowercase{oefficient estimates and \MakeUppercase{H}ankel Determinant for the Class of leaf shaped starlike function}}
	\date{}
	\author[S. P\MakeLowercase{anja}, A. B\MakeLowercase{anerjee and} S. S\MakeLowercase{majumder}]{S\MakeLowercase{hantanu} p\MakeLowercase{anja$^*$}, A\MakeLowercase{bhijit} B\MakeLowercase{anerjee and} S\MakeLowercase{ujoy} M\MakeLowercase{ajumder}}
	\address{Department of Mathematics, University of Kalyani, West Bengal 741235, India.}
	\email{panjasantu07@gmail.com, shantanumath26@klyuniv.ac.in}
	
	\address{ Department of Mathematics, University of Kalyani, West Bengal 741235, India.}
	\email{abanerjee\_kal@yahoo.co.in, abanerjeekal@gmail.com}

	\address{Department of Mathematics, Raiganj University, Raiganj, West Bengal-733134, India.}
	\email{sm05math@gmail.com, sjm@raiganjuniversity.ac.in}
	
	\renewcommand{\thefootnote}{}
	\footnote{2020 \emph{Mathematics Subject Classification}: 30C45, 30C50, 30C55}
	\footnote{\emph{Key words and phrases}:Univalent functions, starlike functions, logarithmic coefficients, coefficient difference, Hankel determinants.}
	\footnote{*\emph{Corresponding Author}: Shantanu Panja.}
	\renewcommand{\thefootnote}{\arabic{footnote}}
	\setcounter{footnote}{0}

		\begin{abstract}
			In this paper, we study a Ma--Minda type subclass of starlike functions associated with the normalized arcsine function $\varphi(z)=1+\frac{2}{\pi}\arcsin z.$
			Using the Carath\'eodory function approach together with sharp coefficient estimates for Schwarz functions, we obtain explicit sharp bounds for the initial Taylor coefficients and logarithmic coefficients of functions in this class. We also derive a sharp estimate for the second Hankel determinant \(H_{2,2}(f)\), as well as for the Hankel determinants \(H_{2,1}(F_f/2)\) and \(H_{2,1}(F_{f^{-1}}/2)\) associated with the logarithmic coefficients of a function and its inverse. In addition, we establish a sharp bound for a difference involving logarithmic coefficients. The extremal functions for all the main results are identified, confirming the sharpness of the estimates.
		\end{abstract}

	\thanks{Typeset by \AmS -\LaTeX}
	\maketitle
	
	\section{{\bf Introduction}}
		Coefficient problems have remained one of the central themes of geometric function theory since the pioneering works of Bieberbach, Carath\'eodory and
	Loewner. Sharp estimates for the Taylor coefficients of analytic functions not	only reveal the geometric behaviour of conformal mappings but also play an
	important role in the study of growth, distortion, covering theorems, Hankel determinants and logarithmic coefficients. Consequently, the search for sharp
	coefficient inequalities has continued to motivate the introduction of new	subclasses of analytic and univalent functions.
	
	Among the various approaches developed over the past decades, differential subordination has proved to be one of the most powerful tools for defining and
	studying geometrically meaningful subclasses. By prescribing suitable first- or higher-order differential constraints, one obtains function classes
	whose analytic structure naturally leads to coefficient estimates, extremal	problems and geometric inclusion relations.
		
Let $\mathcal{A}$ denote the family of analytic functions in the open	unit disk	$	\Omega=\{z\in\mathbb C:|z|<1\},$
normalized by
\begin{equation}\label{e-1.1}
	f(z)=z+\sum_{n=2}^{\infty}a_n z^n.
\end{equation}
The subclass of univalent functions in $\mathcal A$ is denoted by	$\mathcal S$.
A function $f\in\mathcal S$ is said to be starlike if
$f(\Omega)$ is starlike with respect to the origin and convex if $f(\Omega)$ is a convex domain.
The class of starlike functions is characterized analytically by	$$	S^{*}=	\left\{f\in\mathcal S: \Re\!\left(\frac{zf'(z)}{f(z)}\right)>0,\ z\in\Omega
\right\}.$$

We consider $\mathcal{B}_0$ as the class of functions $\omega$ which are analytic in $\Omega$ and satisfies $\omega(0)=0$ and $\mid \omega(z)\mid\le 1$. Every function $\omega\in \mathcal{B}_0$ have a Taylor series expansion $\omega(z)=\sum_{n=1}^{\infty}b_nz^n$ for $z\in\Omega$. Also $\mathcal{B}_0$ is called class of Schwarz functions.
	
	\smallskip
	{\bf\underline{Concept of Subordination:}} Let $f$ and $g$ be two analytic function $\Omega$. The function $f$ is subordinate to $g$, if there exists a Schwarz function $\omega(z)$ such that $f(z)=g(\omega)(z)$ for all $z\in \Omega$ and this is denoted \cite{Lowner_annalen_1923} by $f\prec g$. The concept of subordination used by Ma-Minda \cite{Ma+Minda_1992} in 1992, introduced a broader subclass of $S^*$ and defined by
	\beas &S^*(\varphi)=\bigg\{ f\in\mathcal{A}: \frac{zf^{\prime}(z)}{f(z)}\prec \varphi, \ \ \forall z\in \Omega\bigg\},\eeas 
	where $\varphi$ is an univalent function and satisfying the following conditions:
		\par{(i)} $\Re(\varphi(z))>0$ , for all $z\in \Omega$.
		\par(ii) The range set $\varphi(\Omega)$ being starlike with respect to $\varphi(0)=1$, $\varphi^{\prime}(0)>0$.
		\par(iii) $\varphi(\Omega)$ is symmetric about the real axis.

	In recent years, the concept of subordination has become an important tool in geometric function theory. It has been widely used to introduce and analyze various sub--classes of analytic and univalent functions. In particular Janowski \cite{Janowski_Anpolon_1970} introduced by $\varphi(z)=\frac{1+Az}{1+Bz}$, where $-1\le B< A\le 1$, the classes $S^*(\varphi)$ and $\mathcal{C}(\varphi)$ reduce to the classes $S^{*}[A, B]$ and $\mathcal{C}[A, B]$ respectively. A number of interesting subclasses of starlike and convex functions have been investigated by several authors \cite{Mwndiratta et al_BMMS_2015, Stainkiewicz, Raina_HJMS_2015, Ronning_AUMC_1991, Kumar_AMP_2021, Kumar_Yadav_IJS_2026, cho et al, Bano Raza, Kumar et. al., Alotabi et. al.}. 
	
	% =========================
	% INSERT TABLE HERE
	% =========================
	\begin{table}[H]
		\centering
		\renewcommand{\arraystretch}{1.35}
		\setlength{\tabcolsep}{6pt}
		\begin{tabular}{|c|p{3.2cm}|p{7.7cm}|}
			\hline
			\textbf{Year} & \textbf{Author(s)} & \textbf{Development} \\
			\hline
			1923 & L\"owner \cite{Lowner_annalen_1923} &
			Introduced the concept of subordination in complex analysis. \\
			\hline
			1970 & Janowski \cite{Janowski_Anpolon_1970} &
			Introduced the Janowski classes $\mathcal{S}^{*}[A,B]$ and $\mathcal{C}[A,B]$ via the mapping $\displaystyle\frac{1+Az}{1+Bz}$. \\
			\hline
			1992 & Ma and Minda \cite{Ma+Minda_1992} &
			Unified numerous subclasses by introducing the generalized Ma--Minda subordination framework. \\
			\hline
			Recent years & Various authors \cite{Mwndiratta et al_BMMS_2015,Raina_HJMS_2015,Ronning_AUMC_1991,Kumar_AMP_2021,Kumar_Yadav_IJS_2026,cho et al,Bano Raza,Kumar et. al.,Alotabi et. al.} &
			Development of subclasses associated with various conformal mappings and special functions. \\
			\hline
			Present paper & Authors &
			Introduction of a new Ma--Minda subclass associated with $\displaystyle \varphi(z)=1+\frac{2}{\pi}\arcsin z$, together with coefficient estimates, logarithmic coefficients, and Hankel determinant inequalities. \\
			\hline
		\end{tabular}
		\vspace{.5cc}
		\caption{Evolution of subordination-based subclasses leading to the present work.}
		\label{tab:history}
	\end{table}
		
	Motivated by recent studies on Ma--Minda subclasses associated with special functions, we consider the function  
	$\varphi(z)=1+\frac{2}{\pi}\operatorname{arcsin}z$ for $z\in\Omega$.  Geometrically, $\varphi$ maps $\Omega$ onto a domain shown in the Figure {\ref{fig:im1}, which  is symmetric with respect to the real axis and starlike with respect to $\varphi(0)=1$. The principal branch of the inverse sine function is analytic in $\Omega$ and hence $\varphi$ is analytic and univalent in $\Omega$. Moreover, the image domain $\varphi(\Omega)$ is a vertical strip after normalization, namely
	$\varphi(\Omega)=\left\{\omega\in\mathbb{C}:0<\Re(\omega)<2\right\}.$

	\begin{defi} We introduce the classes
	$$\mathcal{S}_{\arcsin}^{*}=\left\{f\in S:\frac{zf'(z)}{f(z)}\prec 1+\frac{2}{\pi}\operatorname{arcsin}z, \ \forall z\in\Omega\right\},$$
	which represent the subclass of starlike and convex functions associated with the hyperbolic mapping.
	\end{defi}
	\begin{center}
	\begin{figure}[t]
		\centering
		\includegraphics[width=13cm,height=11cm]{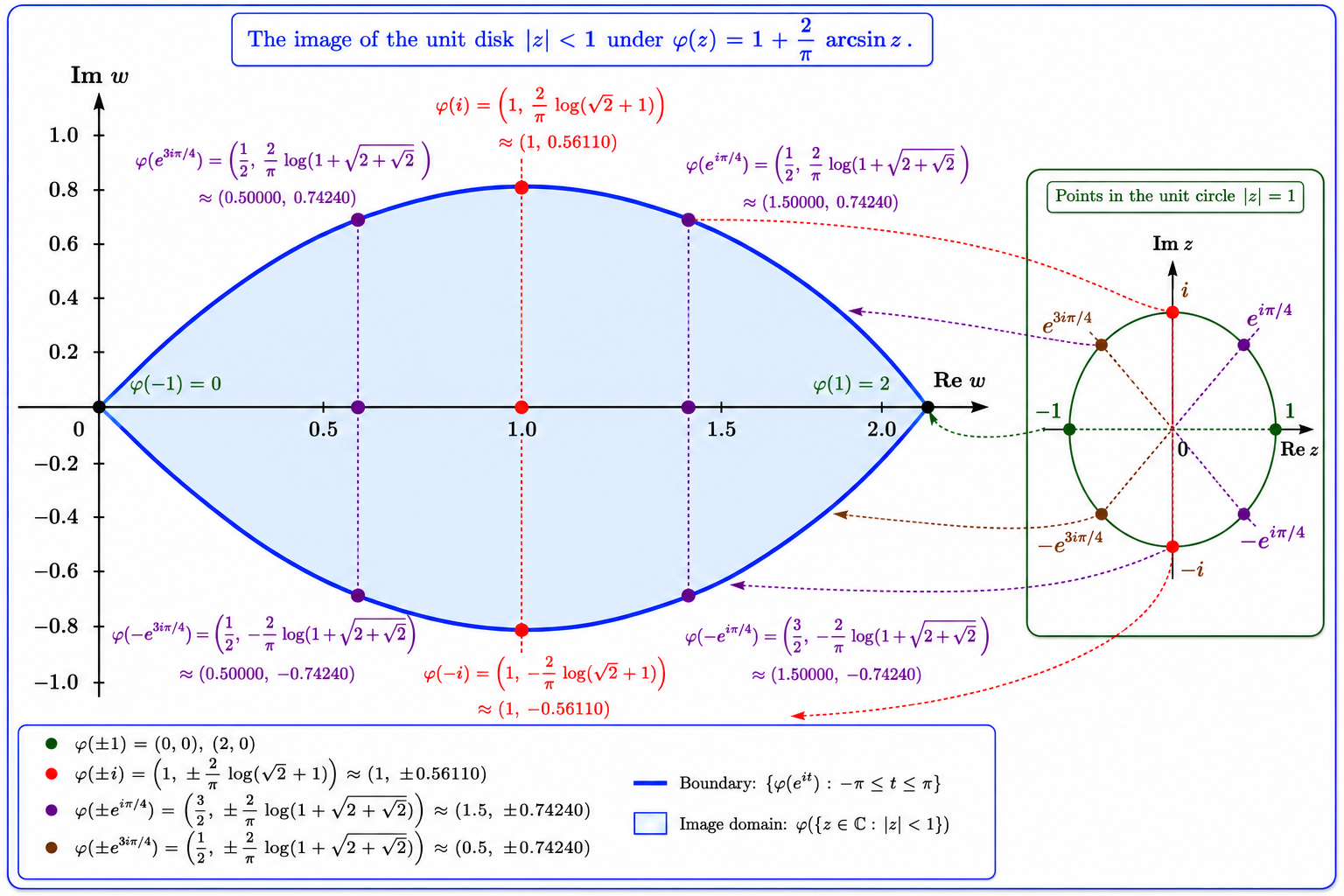}
		\caption{Image of the unit disc $\Omega$ under the normalized mapping $\varphi(z)=1+\dfrac{2}{\pi}\arcsin z$.}
		
		\label{fig:im1}
	\end{figure}
	\end{center}
	{\bf\underline{Background of logarithmic coefficient:}}
	Let a function $f\in S$ and the logarithmic coefficient $\gamma_n$ defined as follows: 
	\bea\label{e1.2} F_f(z)=\log\left({\frac{f(z)}{z}}\right)=2\sum_{n=1}^{\infty}\gamma_n(f)z^n;\ \ z\in\Omega\backslash\{0\} \ \text{with} \ \log 1=0. \eea
	The coefficient $\gamma_n(f)=\gamma_n$ corresponding to each $f\in S$ is known as logarithmic coefficient.
	Differentiating equation (\ref{e1.2}) and comparing coefficients yields relationships that define the coefficients $a_n$, then we get:
	\begin{equation}
		\left\{
		\begin{aligned} 
			\label{e1.3}&\gamma_1=\frac{1}{2}a_2,\\ &\gamma_2=\frac{1}{2}\left(a_3-\frac{1}{2}a_2^2\right),\\ &\gamma_3=\frac{1}{2}\left(a_4-a_2a_3+\frac{1}{3}a_2^3\right),\\
			&\gamma_4=\frac{1}{2}\left(a_5-a_2a_4+a_2^2a_3-\frac{1}{2}a_3^2-\frac{1}{4}a_2^4\right).
		\end{aligned}
		\right.
	\end{equation}
	
	\smallskip
	Logarithmic coefficients play a fundamental role of Milin's Conjecture (see \cite[p. 155]{PL_D_1983}, \cite{Milin_1971}) and have been extensively studied in geometric function theory due to their close relationship with coefficient problems and growth estimates. It is worth noting that sharp estimates are available only for the first few of these coefficients for functions in $S$ are 
	\beas \mid \gamma_1\mid \le  1 \ \text{and} \ \mid \gamma_2\mid \le \frac{1}{2}+\frac{1}{e},\eeas
	whereas for $n\ge 3$, the problem of determining bound for $\mid \gamma_n\mid$ is significantly more difficult and to the best of our knowledge no sharp general bounds for $\mid\gamma_n\mid$ have been established for functions in $S$.
	
	\medskip
	 Let $f \in \mathcal{S}$. The inverse logarithmic coefficients
	\(\Gamma_n\), introduced by Ponnusamy \emph{et al.}~\cite{PSW_RM}, are determined by the following series expansion of the inverse function:
	\[
	F_{f^{-1}}(z):=\log\!\left(\frac{f^{-1}(z)}{z}\right)
	=2\sum_{n=1}^{\infty}\Gamma_n z^n,
	\qquad |z|<\frac{1}{z}.
	\]
	Ponnusamy \emph{et al.}~\cite{PSW_RM} obtained sharp bounds for the inverse logarithmic coefficients of functions in the class \(\mathcal{S}\). They showed that
	\[
	|\Gamma_n(F)|\leq \frac{1}{2n}\binom{2n}{n},
	\]
	with equality holding only for the Koebe function and its rotations. Furthermore, coefficient comparison after differentiating the defining expansion yields explicit formulas for the first few inverse logarithmic coefficients.
	\begin{equation}
		\left\{
		\begin{aligned} 
			\label{eil.4}&\Gamma_1=-\frac{1}{2}a_2,\\ &\Gamma_2=-\frac{1}{2}\left(a_3-\frac{3}{2}a_2^2\right),\\ &\Gamma_3=-\frac{1}{2}\left(a_4-4a_2a_3+\frac{10}{3}a_2^3\right).
			%\\
			%&\Gamma_4=-\frac{1}{2}\left(a_5-5a_2a_4+15a_2^2a_3-\frac{5}{2}a_3^2-\frac{35}{4}a_2^4\right).
		\end{aligned}
		\right.
	\end{equation}
	The study of logarithmic coefficients has been an active area of research in geometric function theory. In particular, numerous sharp estimates have been established for logarithmic coefficients associated with various subclasses of \(\mathcal{S}\) (see \cite{Ali_Vasudevarao_PAMS_2018, Ali_Allu_BAMS, Zaprawa_BSMM_2021, Thomas_PAMS_2016, Ponnusamy_BSM_2021, Ponnusamy_JAMS_2020, Majumder_Indian}). This growing interest has naturally led to the investigation of their inverse counterparts.
	
	\- \smallskip
	{\bf\underline{Hankel determinant:}} For a function $f\in\mathcal{A}$ given by \eqref{e-1.1}, we recall that in 1967 Pommerenke \cite{Pom_1967} introduced the Hankel determinant provide higher-order coefficient invariant of univalent function.
	
	\smallskip
	Let $f\in \mathcal{A}$ be given by \eqref{e-1.1}. The $q$th Hankel determinant $H_{q,n}(f)$, for $q\geq 1$ and $n\geq 0$, is defined by
	\bea\label{a1.4}
	H_{q,n}(f)=
	\begin{vmatrix}
		a_{n} & a_{n+1} & \cdots & a_{n+q-1} \\
		a_{n+1} & a_{n+2} & \cdots & a_{n+q} \\
		\vdots & \vdots & \ddots & \vdots \\
		a_{n+q-1} & a_{n+q} & \cdots & a_{n+2q-2}
	\end{vmatrix}.
	\eea
	In particular for $q=2$ and $n=2$ putting in (\ref{a1.4}) we have $$H_{2,2}(f)=a_2 a_4-a_3^2.$$
	Several authors established sharp bounds for Hankel determinants in important subclasses (see \cite{Raza et. al, Zaprawa, ASS}).
	
	\smallskip 
	For $f\in\mathcal{A}$ be given by \eqref{e-1.1}. The $q$th Hankel determinant $H_{q, n}{(F_f/2)}$, for $q\geq 1$ and $n\geq 0$ is defined for the logarithmic coefficient as follows:
	\bea\label{b1.6}
	H_{q, n}(F_f/2)=
	\begin{vmatrix}
		\gamma_{n} & \gamma_{n+1} & \cdots & \gamma_{n+q-1} \\
		\gamma_{n+1} & \gamma_{n+2} & \cdots & \gamma_{n+q} \\
		\vdots & \vdots & \ddots & \vdots \\
		\gamma_{n+q-1} & \gamma_{n+q} & \cdots & \gamma_{n+2q-2}
	\end{vmatrix}.
	\eea
	In particular for $q=2$ and $n=1, 2$ putting in (\ref{b1.6}), we have \beas H_{2, 1}(F_{f}/2)=\gamma_1\gamma_3-\gamma_2^2.\eeas.
	
	\medskip
	 \smallskip 
	For $f\in\mathcal{A}$ be given by \eqref{e-1.1}. The $q$th Hankel determinant $H_{q, n}{(F_{f^{-1}}/2)}$, for $q\geq 1$ and $n\geq 0$ is defined for the logarithmic coefficient as follows:
	\bea\label{e1.4}
	H_{q, n}(F_{f^{-1}}/2)=
	\begin{vmatrix}
		\Gamma_{n} & \Gamma_{n+1} & \cdots & \Gamma_{n+q-1} \\
		\Gamma_{n+1} & \Gamma_{n+2} & \cdots & \Gamma_{n+q} \\
		\vdots & \vdots & \ddots & \vdots \\
		\Gamma_{n+q-1} & \Gamma_{n+q} & \cdots & \Gamma_{n+2q-2}
	\end{vmatrix}.
	\eea
	
	In particular for $q=2$ and $n=1, 2$ putting in (\ref{e1.4}), we have \beas H_{2, 1}(F_{f^{-1}}/2)=\gamma_1\gamma_3-\gamma_2^2.\eeas.
	The investigation of Hankel determinant whose entries consist of logarithmic coefficients of analytic functions was initiated by Kowalczyk--Lecko \cite{KL_BAMS_2022, KL_RACSAM_2023}. This framework has provided a powerful tool for studying coefficient-related problems and has significantly enriched the theory of univalent functions. Consequently, it has become an active area of research, with numerous contributions appearing in the literature; (see \cite{MK_BIMS_2023, Lecko et al, Majumder et al._JA_2026, S_Arif_Lithunian, Kumar_2024, BBS_2024, Nagpal_2025, Raza_Riaz}) for further developments.
	
\section{\bf{Lemmas}}
		Let $\mathcal{P}$ be the class of all analytic functions in the unit disk $\Omega$ such that $p(0)=1$ and $\Re(p(z))>0$ for all $z\in \Omega$. Every $p\in \mathcal{P}$ then has the series representation 
		\bea\label{e2.1} p(z)=1+\sum_{n=1}^{\infty} c_nz^n, \ \ z\in\Omega. \eea
		Function in $\mathcal{P}$ are referred to as Carath\'{e}odory functions. It will known that for $p\in\mathcal{P}$, the coefficients $\left(\text{see \cite{PL_D_1983}}\right)$ satisfy the sharp bound $\mid c_n\mid\leq 2$ for all $n\ge 1$. The Carath\'{e}odory class $\mathcal{P}$ and its coefficient bounds play a fundamental  role in deriving sharp estimates in geometric function theory.
	%	\begin{lem}\label{le.1}\cite{Grenander_1958}
	%		If $p\in\mathcal{P}$ be given by \emph{(\ref{e2.1})}, then
	%		\beas &&2c_2=c_1^2+x(4-c_1^2),\\ && 4c_3=c_1^3+2(4-c_1^2)c_1x-c_1(4-c_1^2)x^2+2(4-c_1^2)(1-\mid x\mid^2)z,\eeas
	%	\end{lem}
	\begin{lem}\label{l2.1}\cite{Ma+Minda_1992}
		If $p\in\mathcal{P}$ be given by \emph{(\ref{e2.1})}, then
		\[
		\mid c_2-{v} c_1^2\mid \le 
		\begin{cases}
			-4{v}+2 & \text{if } {v}<0, \\
			2 & \text{if } 0\le{v}\le 1,\\
			4{v}-2 & \text{if } {v}>1.
		\end{cases}
		\]
		Moreover, for ${v}<0$ or ${v}>1$ equality holds if and only if $h(z)=\frac{1+z}{1-z}$ one of its rotations. \\Also, for $0<{v}<1$ equality holds if and only if $h(z)=\frac{1+z^2}{1-z^2}$ or one of its rotations. 
	\end{lem}
	\begin{lem}\label{l2.2}\cite{Ali_BMSS_2003}
		If $p\in\mathcal{P}$ be given by \emph{(\ref{e2.1})} with $0\le B\le 1$ and $2B(2B-1)\le D\le B$. Then 
		\beas \mid c_3-2Bc_1c_2+Dc_1^3\mid\le2.\eeas.
	\end{lem}
	\begin{lem}\label{l2.3}\cite{Ravichandran_Verma_CRMAA_2015}
		If $p\in\mathcal{P}$ be given by \emph{(\ref{e2.1})}. If $\beta$, $\gamma$, $\delta$ and $\xi$ satisfy $0<\beta<1$, $0<\xi<1$, and 
		\beas 8\xi(1-\xi)\big\{(\beta\gamma-2\delta)^2+(\beta(\xi+\beta)-\gamma)^2\big\}+\beta(1-\beta)(\gamma-2\beta\xi)^2\le 4\beta^2(1-\beta)^2\xi(1-
		\xi),\eeas
		then \beas\mid\delta c_1^4+\xi c_2^2+2\beta c_1c_3-\frac{3}{2}\gamma c_1^2c_2-c_4 \mid\le 2. \eeas
	\end{lem}
	\begin{lem}\label{l2.4} \cite[Lemma 2.4]{CKL1} If $p\in\mathcal{P}$ is of the form \emph{(\ref{e2.1})}, then
		\bea
		\label{e2.2}c_1 =2\tau_1,\eea
		\bea\label{e2.3} c_2=2\tau_1^2 + 2(1 - \tau_1^2)\tau_2\eea
		and
		\bea
		\label{e2.4} c_3 = 2\tau_1^3+4(1-\tau_1^2)\tau_1\tau_2 - 2(1 - \tau_1^2)\tau_1\tau_2^2 + 2(1 - \tau_1^2)(1 - |\tau_2|^2)\tau_3
		\eea
		for some $\tau_1, \tau_2, \tau_3 \in\ol\Omega:= \{z \in \mathbb{C}: |z| \leq 1 \}$.
		
		\medskip
		For $\tau_1 \in \mathbb{T}:= \{z \in \mathbb{C}: |z| = 1 \}$, there is a unique function $p\in \mathcal{P}$ with $c_1$  as in \emph{(\ref{e2.2})}, namely
		\[p(z)=\frac{1+\tau_1 z}{1 - \tau_1 z}, \quad z \in \Omega.\]
		
		\medskip
		For $\tau_1 \in \Omega$ and $\tau_2 \in \mathbb{T}$, there is a unique function $p\in \mathcal{P}$ with $c_1$ and $c_2$ as in \emph{(\ref{e2.2})} and \emph{(\ref{e2.3})}, namely
		\[p(z) = \frac{1+(\ol \tau_1 \tau_2+\tau_1)z + \tau_2 z^2}{1+(\ol \tau_1 \tau_2 - \tau_1)z - \tau_2 z^2}, \quad z \in \Omega.\]

		\medskip
		For $\tau_1, \tau_2 \in \Omega$ and $\tau_3 \in \mathbb{T}$, there is a unique function $p\in\mathcal{P}$ with $c_1$, $c_2$ and $c_3$  as in \emph{(\ref{e2.2})--(\ref{e2.4})}, namely
		\[p(z)=\frac{1 + (\ol\tau_2 \tau_3 + \ol\tau_1 \tau_2 + \tau_1)z+(\ol\tau_1\tau_3+\tau_1\ol\tau_2\tau_3+\tau_2)z^2+\tau_3z^3}{1+(\ol\tau_2\tau_3+\ol\tau_1\tau_2-\tau_1)z+(\ol\tau_1\tau_3-\tau_1\ol\tau_2\tau_3-\tau_2)z^2-\tau_3z^3},\;\;z\in\Omega.\]
	\end{lem}

	\medskip
	Following well-known result is due to Choi et al. \cite{CKS1}.
	\begin{lem}\label{l2.5}\cite{CKS1} Let $A$, $B$, $C$  be real numbers and let
		\[\Psi(A, B, C):= \max\limits_{z\in \ol{\mathbb{D}}}\left\lbrace |A+Bz+Cz^2|+1-|z|^2\right\rbrace.\]
		
		\begin{enumerate} 
			\item[\emph{(i)}] If $AC\geq 0$, then
			\[\Psi(A, B, C) =
			\begin{cases}
				|A|+|B|+|C|, & \text{if}\;\;\; |B|\geq 2(1-|C|), \\
				1+|A|+\frac{B^2}{4(1-|C|)}, &\text{if}\;\;\; |B|<2(1-|C|).
			\end{cases}
			\]
			\item[\emph{(ii)}] If $AC<0$, then 
			\[\Psi(A,B,C)=
			\begin{cases}
				1-|A|+\frac{B^2}{4(1-|C|)}, &\text{if}\;\;\; -4AC(C^{-2}-1) \leq B^2\; \text{and}\; |B|<2(1-|C|), \\
				1+|A|+\frac{B^2}{4(1+|C|)}, &\text{if}\;\;\; B^2<\min\left\{4(1+|C|)^2, -4AC(C^{-2}-1) \right\}, \\
				R(A,B,C), &\text{otherwise},
			\end{cases}
			\]
			where
			\[R(A,B,C):=
			\begin{cases}
				|A|+|B|-|C|, & \text{if}\;\;\; |C|(|B|+4|A|) \leq |AB|, \\
				-|A|+|B|+|C|, & \text{if}\;\;\; |AB|\leq |C|(|B|-4|A|), \\
				(|C|+|A| )\sqrt{1-\frac{B^2}{4AC}}, &\text{otherwise}.
			\end{cases}
			\]
		\end{enumerate} 
	\end{lem}
	\begin{lem}\label{l2.6}\cite{Sim_Thomas_Symmetry_2020}
		Let $J$, $K$, and $L$ be numbers such that $J\geq 0$, $K\in\mathbb{C}$, and
		$L\in\mathbb{R}$. Let $p\in\mathcal{P}$ be of the form \eqref{e2.1} and define
		a function by
		\[
		\Phi(c_1,c_2)
		=
		\left|Kc_1^{2}+Lc_2\right|
		-
		J|c_1|.
		\]
		
		Then
		\[
		\Phi(c_1,c_2)
		\le
		\begin{cases}
			|4K+2L|-2J,
			& \text{if } |2K+L|\ge |L|+J,\\[2mm]
			2|L|,
			& \text{otherwise}.
		\end{cases}
		\]
		
		and
		
		\[
		-\Phi(c_1,c_2)
		\le
		\begin{cases}
			2J-M,
			& \text{when } J\ge M+2|L|,\\[3mm]
			2J\sqrt{\dfrac{2|L|}{M+2|L|}},
			& \text{when } J^{2}\le 2|L|(M+2|L|),\\[4mm]
			2|L|+\dfrac{J^{2}}{M+2|L|},
			& \text{otherwise},
		\end{cases}
		\]
		
		where
		\[
		M=|4K+2L|.
		\]
	\end{lem}
	\begin{tcolorbox}[
		enhanced,
		breakable,
		colback=yellow!6!white,
		colframe=red!65!black,
		boxrule=1pt,
		arc=2mm,
		left=1.5mm,
		right=1.5mm,
		top=1.2mm,
		bottom=1.2mm,
		title=\textbf{Motivation and Main Results},
		fonttitle=\bfseries\large,
		coltitle=black,
		attach boxed title to top left={xshift=2mm,yshift=-2mm},
		boxed title style={
			colback=red!12!white,
			colframe=red!65!black,
			sharp corners,
			boxrule=0.6pt
		}
		]
		\textit{Motivated by the need for sharp and structurally meaningful estimates in geometric function theory, we investigate the Ma--Minda-type class associated with the normalized $\arcsin$ function.}
		The study of the class $\mathcal{S}^{*}_{\arcsin}$ is natural because the mapping
		$\varphi(z)=1+\frac{2}{\pi}\arcsin z$
		combines geometric regularity with analytic flexibility. This makes it a suitable candidate for deriving sharp bounds for coefficients, logarithmic coefficients, and Hankel determinants. The results obtained here provide precise estimates for $a_n$, $\gamma_n$, $\Gamma_n$, and the associated Hankel determinants to the related class $\mathcal{S}^{*}_{\arcsin}$.
	\end{tcolorbox}
	\section{\bf{Sharp bounds of Coefficients for the class $\mathcal{S}^{*}_{\arcsin}$}}
	\begin{theo}\label{t3.1}
			Let $f(z)=z+a_2z^2+a_3z^3+\cdots\in\mathcal{S}^{*}_{\arcsin} $. Then
		$$
		|a_2|\le \frac{2}{\pi}, \ \mid a_3\mid \le \frac{1}{\pi}, \ \mid a_4\mid \le \frac{2}{3\pi}.
		$$
		All inequalities are sharp.
	\end{theo}
	
	\begin{proof}
		Let $f\in \mathcal{S}^{*}_{\arcsin}$. Then, by the definition of the class, there exists a Schwarz function $\omega$ satisfying
		\bea\label{e3.1} \frac{zf^{\prime}(z)}{f(z)}=1+\frac{2}{\pi}\operatorname{arcsin}\omega(z). \eea 
		
		Suppose that $\omega(z)=\frac{p(z)-1}{p(z)+1}$, where $p\in\mathcal{P}$ is defined in (\ref{e2.1}), then we have 
		\bea\label{e3.2} 
		\omega(z)=\frac{c_1}{2}z+\left(\frac{c_2}{2}-\frac{c_1^2}{4}\right)z^2+\left(\frac{c_3}{2}-\frac{c_1c_2}{2}+\frac{c_1^3}{8}\right)z^3+\left(\frac{c_4}{2}-\frac{c_1c_3}{2}-\frac{c_2^2}{4}+\frac{3c_1^2c_2}{8}-\frac{c_1^4}{16}\right)z^4+\cdots.
		\eea
		Substituting (\ref{e-1.1}) and (\ref{e3.2}) into (\ref{e3.1}) and comparing the coefficients of like powers of $z$, we obtain 
		\bea 
		\label{e3.3}a_2&=&\frac{1}{\pi}c_1,\\ 
		\label{e3.4}a_3&=&\frac{1}{2\pi}\left(c_2-\left(\frac{\pi-2}{2\pi}\right)c_1^2\right),\\
		\label{e3.5}a_4&=&\frac{1}{3\pi}\left(c_3-\left(\frac{2\pi-3}{2\pi}\right)c_1c_2+\left(\frac{7\pi^2-18\pi+12}{24\pi^2}\right)c_1^3\right),\\
		\label{e3.6}a_5&=&\frac{1}{4\pi}\bigg(c_4-\left(\frac{\pi-1}{2\pi}\right)c_2^2-\left(\frac{3\pi-4}{3\pi}\right)c_1c_3+\left(\frac{21\pi^2-44\pi+24}{24\pi^2}\right)c_1^2c_2\nonumber\\&&-\left(\frac{27\pi^3-74\pi^2+72\pi-24}{144\pi^3}\right)c_1^4\bigg).\eea
		{\bf{A.}} It follows immediately from (\ref{e3.3})  that
		$ \mid a_2\mid\le \frac{2}{\pi}$.
		\par To show the sharpness of the estimate, we consider $p(z)=\frac{1+z}{1-z}$. If $f\in \mathcal{S}^{*}_{{\arcsin}}$, then its series expansion:
		\bea\label{f1} 
		f_1(z)=z+\frac{2}{\pi} z^2+\frac{2}{\pi^2}z^3+\left(\frac{1}{9\pi}+\frac{4}{3\pi^3}\right)z^4+\cdots.\eea
		{\bf{B.}} Applying Lemma~\ref{l2.1} to (\ref{e3.4}), we obtain
		\beas \mid a_3\mid = \frac{1}{2\pi}\left\mid c_2-\left(\frac{\pi-2}{2\pi}\right)c_1^2\right\mid \le \frac{1}{\pi}.\eeas
		\par For sharpness, we choose $p(z)=\frac{1+z^2}{1-z^2}$. If $f\in \mathcal{S}^{*}_{{\arcsin}}$, then its series expansion:
		\bea\label{f2} 
		f_2(z)=z+\frac{1}{\pi} z^3+\frac{1}{2\pi^2}z^5+\left(\frac{3+\pi^2}{18\pi^3}\right)z^7+\cdots.\eea
		{\bf{C.}} Here equation (\ref{e3.5}) can be rewritten as 
		\bea\label{c3.7} 
		\mid a_4\mid& =&\frac{1}{3\pi}\left\mid c_3-\left(\frac{2\pi-3}{2\pi}\right)c_1c_2+\left(\frac{7\pi^2-18\pi+12}{24\pi^2}\right)c_1^3\right\mid\nonumber\\&=&\frac{1}{3\pi}\mid c_3-2Bc_1c_2+Dc_1^3 \mid,
		\eea
		where $B=\left(\frac{2\pi-3}{4\pi}\right)\approx 0.26127$ and $D=\left(\frac{7\pi^2-18\pi+12}{24\pi^2}\right)\approx0.10359$. Since
		$2B(2B-1)\le D\le B,$
		Lemma~\ref{l2.2} yields 
		\beas \mid a_4\mid \le \frac{2}{3\pi}.\eeas
		\par For sharpness of the inequality, we take $p(z)=\frac{1+z^3}{1-z^3}$. If $f\in \mathcal{S}^{*}_{{\arcsin}}$, then its series expansion:
		\bea\label{f3} 
		f_3(z)=z+\frac{2}{3\pi} z^4+\frac{2}{9\pi^2}z^7+\left(\frac{4+3\pi^2}{81\pi^3}\right)z^{10}+\cdots.\eea 
		This completes the proof.
	\end{proof}
	\begin{theo}\label{t3.2}
			Let $f(z)=z+a_2z^2+a_3z^3+\cdots\in\mathcal{S}^{*}_{\arcsin} $. Then Hankel determinant
		$$
		|H_{2, 2}(f)|\le \frac{1}{\pi^2}.
		$$
	\end{theo}
	\begin{proof}
		 Since $f\in \mathcal{S}^{*}_{\arcsin}$, then we know the second order Hankel determinant 
		\bea\label{a3.11} \mid H_{2,2}(f)\mid&=&\mid a_2a_4-a_3^2\mid\nonumber\\&=& \frac{1}{3\pi^2}\left\mid c_1c_3-\frac{1}{4}c_1^2c_2-\frac{3}{4}c_2^2+\left(\frac{5\pi^2-12}{48\pi^2}\right)c_1^4 \right\mid
		\eea
		Substituting the Libera--Złotkiewicz representations
		(\ref{e2.2})--(\ref{e2.4}) into (\ref{a3.11}), we obtain
		\bea\label{a3.12}
		 \mid H_{2,2}(f)\mid&=&\frac{1}{3\pi^2}\left\mid \frac{2(\pi^2-6)}{3\pi^2}\tau_1^4-(1-\tau_1^2)(3+\tau_1^2)\tau_2^2+4(1-\tau_1^2)(1-\mid\tau_2\mid^2)\tau_1\tau_3\right\mid.
		\eea
		As $c_1\in [0, 2]$, then by Lemma~\ref{l2.4} we have $\tau_1\in[0, 1]$. Therefore from (\ref{a3.12}) it follows that 
		\beas \mid H_{2,2}(f)\mid=
		\begin{cases}
			\frac{1}{\pi^2}\mid\tau_2\mid^2\le \frac{1}{\pi^2}\approx 0.10132, & \tau_1=0, \\[3mm]
			\frac{2(\pi^2-6)}{9\pi^4}\approx 0.00883, & \tau_1=1.
		\end{cases} \eeas
		For $0<\tau_1<1$, applying the triangle inequality to
		(\ref{a3.12}), we obtain the following estimation, for $0<\tau_1<1$ and $|\tau_3|\le1$.
		\bea\label{a3.13} 
		\mid H_{2,2}(f)\mid&\leq&\frac{1}{3\pi^2}\left(\left\mid \frac{2(\pi^2-6)}{3\pi^2}\tau_1^4-(1-\tau_1^2)(3+\tau_1^2)\tau_2^2\right\mid+\left\mid4(1-\tau_1^2)(1-\mid\tau_2\mid^2)\tau_1\tau_3\right\mid\right)\nonumber\\&\le&\frac{4}{3\pi^2}\tau_1(1-\tau_1^2)\left(\left\mid \frac{(\pi^2-6)\tau_1^3}{6\pi^2(1-\tau_1^2)}-\left(\frac{3+\tau_1^2}{4\tau_1}\right)\tau_2^2\right\mid+(1-\mid\tau_2\mid^2)\right)\nonumber\\&\le& \frac{4}{3\pi^2}\mid\tau_1\left(1-\tau_1^2\right)\mid\Psi\left(A_1, B_1, C_1\right),
		\eea
		where $\Psi\left(A_1, B_1, C_1\right)=\mid A_1+B_1\tau_2+C_1\tau_2^2\mid+1-\mid\tau_2\mid^2$ with \beas A_1=\frac{(\pi^2-6)\tau_1^3}{6\pi^2(1-\tau_1^2)}, \  B_1=0 \ \text{and} \ C_1=-\frac{3+\tau_1^2}{4\tau_1}.\eeas
		\smallskip
		We now consider the following cases under Lemma \ref{l2.5}, according to the various choice of $A_1$, $B_1$ and $C_1$.
		\smallskip
		
		Suppose $\tau_1\in (0, 1)$. Clearly, $A_1C_1\le 0$, also we have
		\bea\label{s1} 
		\left(\frac{1}{C_1^2}-1\right)=-\frac{(9-\tau_1^2)(1-\tau_1^2)}{\left(3+\tau_1^2\right)^2}<0, \ \text{for} \ 0<\tau_1<1.
		\eea
		Hence $-4A_1C_1(C_1^{-2}-1)>0$, also $4(1+\mid C_1\mid)>0$ for all $0<\tau_1<1$. \\Consequently $B_1^2<\min\{-4A_1C_1(C_1^{-2}-1), 4(1+\mid C_1\mid)\}$.
		Hence, by Lemma {\ref{l2.5}} we get
		 
		\bea\label{s2} \Psi\left(A_1, B_1, C_1\right) =1+ \mid A_1\mid+\frac{B_1^2}{4(1+\mid C_1\mid)}.\eea
		Now from (\ref{a3.13}) and (\ref{s2}) we get
		\bea\label{a3.16} \mid H_{2,2}(f)\mid&\le&\frac{2}{9\pi^4}\left((\pi^2-6)\tau_1^4-6\pi^2\tau_1^3+6\pi^2\tau_1\right)=\frac{2}{9\pi^4}P_1(\tau_1),\eea
		where $P_1(\tau_1)=(\pi^2-6)\tau_1^4-6\pi^2\tau_1^3+6\pi^2\tau_1$.
	
	\begin{figure}[H]
		\centering
		\begin{tikzpicture}
			
			\begin{axis}[
				width=12cm,
				height=8cm,
				xmin=0,
				xmax=1.03,
				domain=0:1,
				samples=300,
				axis lines=left,
				xlabel={$\tau_1$},
				ylabel={$P_1(\tau_1)$},
				xlabel style={
					font=\large,
					yshift=-8pt
				},
				ylabel style={font=\large},
				tick label style={font=\large},
				enlargelimits=false,
				clip=false,
				ymax=26.5,
				ymin=0,
				grid=major,
				grid style={gray!55},
				]
				
				% Polynomial
				\addplot[
				blue,
				very thick
				]
				{(pi^2-6)*x^4-6*pi^2*x^3+6*pi^2*x};
				
				% Maximum point
				\addplot[
				only marks,
				mark=*,
				mark size=2.5pt,
				red
				]
				coordinates {(0.59287,23.2987)};
				
				% Dashed guides
				\draw[
				red,
				line width=1.2pt,
				dash pattern=on 6pt off 3pt
				]
				(axis cs:0.59287,0)
				--
				(axis cs:0.59287,23.2987);
				
				\draw[
				red,
				line width=1.2pt,
				dash pattern=on 6pt off 3pt
				]
				(axis cs:0,23.2987)
				--
				(axis cs:0.59287,23.2987);
				
				% Annotation
				\node[
				fill=white,
				inner sep=1.5pt,
				anchor=west,
				font=\small
				]
				at (axis cs:0.41,19.55)
				{$(\tau_0,P_1(\tau_0))$};
				
				\node[
				font=\small,
				anchor=north
				]
				at (axis cs:0.59287,-1.85)
				{$\tau_0\approx0.59287$};
			\draw[
			->,
			very thick,
			gray!80!black
			]
			(axis cs:0.47,20.8)
			--
			(axis cs:0.57287,22.8987);	
			\end{axis}
			
		\end{tikzpicture}
		
		\caption{The polynomial
			$P_1(\tau_1)=(\pi^2-6)\tau_1^4-6\pi^2\tau_1^3+6\pi^2\tau_1$ on $0\le\tau_1\le1$. The unique maximum occurs at $\tau_0\approx0.59287$.}
		\label{fig:P1}
	\end{figure}
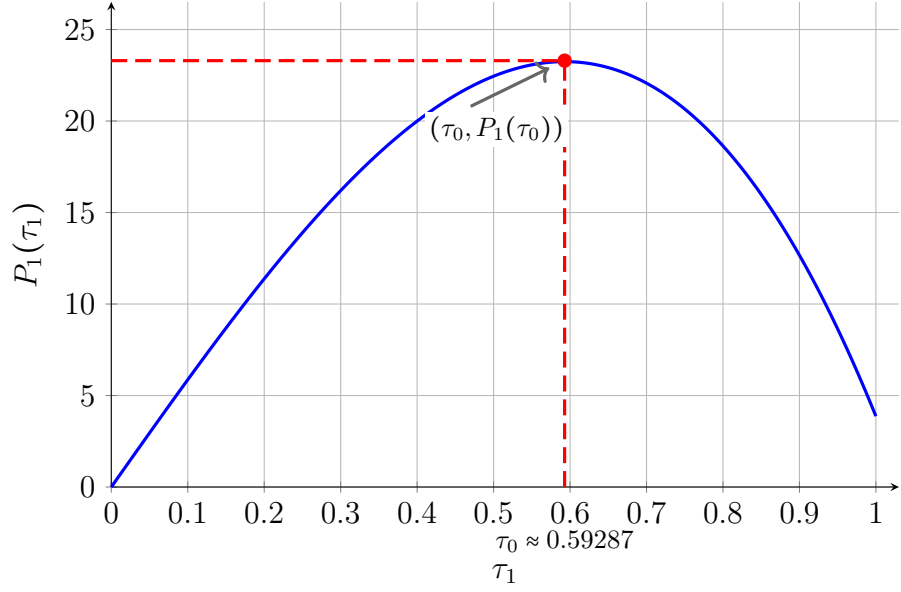
		
	Differentiating, $P_1'(\tau_1)=4(\pi^2-6)\tau_1^3
	-18\pi^2\tau_1^2+6\pi^2.$
	The equation $P_1'(\tau_1)=0$ has a unique solution
	$\tau_0\approx0.59287$ in the interval $(0,1)$.
	Moreover, $P_1''(\tau_0)<0$,	which shows that $P_1(\tau_1)$ attains its maximum at
	$\tau_1=\tau_0$ (see Figure~\ref{fig:P1}).
	
	Now from (\ref{a3.16}) we obtain
	\beas
	|H_{2,2}(f)|
	\le
	\frac{2}{9\pi^4}P_1(0.59287)
	\approx0.05303.
	\eeas	
		  
		Next from (\ref{a3.16}) we obtain
		\beas \mid H_{2,2}(f)\mid&\le&\frac{2}{9\pi^4}P_1(0.59287)\approx 0.05303.\eeas
		Combining the estimates for the cases
		$\tau_1=0$, $\tau_1=1$, and $0<\tau_1<1$, we conclude that
		$$|H_{2,2}(f)|\le\frac1{\pi^2}.$$
		The estimate is sharp, since equality is attained by the function $f_2$ defined in (\ref{f2}). 
		
	\end{proof}

	\section{\bf{Sharp bounds of Coefficients and logarithmic coefficients for the class $\mathcal{S}^{*}_{\arcsin}$}}
       \begin{theo}
	 	Let $f(z)=z+a_2z^2+a_3z^3+\cdots\in\mathcal{S}^{*}_{\arcsin} $ and $\gamma_1$, $\gamma_2$, $\gamma_3$, $\gamma_4$ are given in \emph{(\ref{e1.3})}. Then
	 	$$
	 	|\gamma_n|\le \frac{1}{n\pi},
	 	\qquad n=1,2,3,4.
	 	$$
	 	The estimates are sharp.  \end{theo}
	 \begin{proof} 
	 	Employing the same procedure as in the proof of Theorem~\ref{t3.1}, the coefficients
	 	$a_2,a_3,a_4$ and $a_5$ can be expressed in terms of the coefficients of the
	 	Carath\'eodory function $p(z)=1+c_1z+c_2z^2+\cdots$ as given in
	 	\eqref{e3.3}--\eqref{e3.6}.\\
	 	{\bf{A.}} From (\ref{e1.3}) and (\ref{e3.3}), we get 
	 	\beas \mid\gamma_1\mid =\frac{1}{2\pi}\mid c_1\mid\le\frac{1}{\pi}.\eeas
	 	Hence, $$|\gamma_1|\le\frac1{\pi},$$
	 	which is the desired estimate.
	 	Equality is attained by the extremal function $f_1$ defined in \eqref{f1}.\\
	 	{\bf{B.}} Using (\ref{e1.3}) and (\ref{e3.3})--(\ref{e3.4}) we obtain
	 	\beas \mid \gamma_2 \mid &=& \frac{1}{2}\mid a_3-\frac{1}{2}a_2^2\mid\\&=&\frac{1}{4\pi}\mid c_2-\frac{1}{2}c_1^2\mid.\eeas
	 	Therefore, by Lemma \ref{l2.1} we get $\mid \gamma_2 \mid\le \frac{1}{2\pi}$.\\
	 	To verify the sharpness of the above result, we employ the function $f_2$ introduced in \eqref{f2}.\\
	 	{\bf{C.}} In this case using (\ref{e3.3})--(\ref{e3.5}) in (\ref{e1.3}), then we get
	 	\beas 
	 	\mid\gamma_3\mid&=&\frac{1}{2}\mid a_4-a_2a_3+\frac{1}{3}a_2^3\mid\\ &=&
	 	\frac{1}{6\pi}\mid c_3-c_1c_2+\frac{7}{24}c_1^3\mid\\&=&\frac{1}{6\pi}\mid c_3-2Bc_1c_2+Dc_1^3 \mid,
	 	\eeas
	 	where $B=\frac{1}{2}$ and $D=\frac{7}{24}$. Since $B=\frac12$ and $D=\frac{7}{24}$ satisfy the hypotheses of Lemma~\ref{l2.2}, it follows that
	 	\beas \mid \gamma_3\mid \le \frac{1}{3\pi}.\eeas
	 	The above result is sharp with the function $f_3$ defined by (\ref{f3}).\\
	 	{\bf{D.}} Using (\ref{e3.3})--(\ref{e3.6}) in (\ref{e1.3}), then we obtain
	 	\bea\label{e3.8} \mid\gamma_4\mid&=&\frac{1}{2}\mid a_5-a_2a_4+a_2^2a_3-\frac{1}{2}a_3^2-\frac{1}{4}a_2^4\mid\nonumber\\&=&
	 	\frac{1}{8\pi}\mid \frac{3}{16}c_1^4 +\frac{1}{2}c_2^2+c_1c_3-\frac{7}{8}c_1^2c_2-c_4\mid\nonumber\\&=&\frac{1}{8\pi}\mid\delta c_1^4+\xi c_2^2+2\beta c_1c_3-\frac{3}{2}\gamma c_1^2c_2-c_4 \mid,\eea
	 	where $\delta=\frac{3}{16}$, $\xi=\beta=\frac{1}{2}$ and $\gamma=\frac{7}{12}$. Observe that $\beta=\xi=\frac12\in(0,1)$,
	 	so the first hypothesis of Lemma~\ref{l2.3} is satisfied. Also we deduce that:
	 	$$ 8\xi(1-\xi)\big\{(\beta\gamma-2\delta)^2+(\beta(\xi+\beta)-\gamma)^2\big\}+\beta(1-\beta)(\gamma-2\beta\xi)^2-4\beta^2(1-\beta)^2\xi(1-
	 	\xi)=-\frac{19}{576}<0. $$
	 	Hence all the hypotheses of Lemma~\ref{l2.3} are satisfied. Applying the lemma to \eqref{e3.8} yields 
	 	\beas \mid \gamma_4\mid \le \frac{1}{4\pi}.\eeas
	 		\par For sharpness, we choose $p(z)=\frac{1+z^4}{1-z^4}$. The corresponding extremal function has the expansion
	 	\bea\label{f4} 
	 	f_4(z)=z+\frac{1}{2\pi} z^5+\cdots.\eea
	 \end{proof}
	 \begin{theo}
	 	Let $f(z)=z+a_2z^2+a_3z^3+\cdots\in\mathcal{S}^{*}_{\arcsin}$. Then
	 	\beas\mid H_{2, 1}(F_{f}/2)\mid \ \le \ \frac{1}{4\pi^2}. \eeas
	 	The inequality is sharp.
	 \end{theo}
	 \begin{proof}
	 	Since $f\in S^{*}_{\arcsin}$, we know that 
	 	\bea\label{m3.9} H_{2, 1}(F_{f}/2)=\gamma_1\gamma_3-\gamma_2^2.\eea
	 	Using (\ref{e1.3}), (\ref{e3.3})--(\ref{e3.5}) into (\ref{m3.9}) and we obtain the expression:
	 	\bea\label{m3.10} H_{2,1}(F_{f}/2)&=&\frac{1}{4}\left(a_2a_4-a_3^2+\frac{1}{12}a_2^4\right)\nonumber\\ &=&\frac{1}{576\pi^2}\left(48c_1c_3-36c_2^2-12c_1^2c_2+5c_1^4\right).\eea
	 	Substituting the Carath\'eodory coefficient representations \eqref{e2.2}--\eqref{e2.4} into \eqref{m3.10} and simplifying, we obtain 
	 	\bea\label{m3.11} 
	 	\mid H_{2,1}(F_{f}/2)\mid =\frac{1}{36\pi^2}\bigg\mid 2\tau_1^4-3\left(1-\tau_1^2\right)\left(3+\tau_1^2\right)\tau_2^2+12\left(1-\tau_1^2\right)\left(1-\mid\tau_2\mid^2\right)\tau_1\tau_3\bigg\mid.
	 	\eea
	 	Since $\tau_1\in[0,1]$ by Lemma~\ref{l2.4}, we first consider the boundary cases:
	 	\beas \mid H_{2,1}(F_{f}/2)\mid=
	 	\begin{cases}
	 		\frac{1}{4\pi^2}\mid\tau_2\mid^2\le \frac{1}{4\pi^2}\approx 0.02533, & \tau_1=0 \\
	 		\frac{1}{18\pi^2}\approx 0.00563, & \tau_1=1.
	 	\end{cases} \eeas 
	 	Let us consider the case $0<\tau_1<1$.
	 	Also, by applying the triangle inequality to (\ref{m3.11}), for $\tau_1\in (0, 1)$ and $\mid \tau_3\mid\le1$, we obtain the following inequality
	 	\bea\label{m3.12}
	 	\mid H_{2,1}(F_{f}/2)\mid&\le& \frac{1}{36\pi^2}\big\mid2\tau_1^4-3\left(1-\tau_1^2\right)\left(3+\tau_1^2\right)\tau_2^2\big\mid+\big\mid12\left(1-\tau_1^2\right)\left(1-\mid\tau_2\mid^2\right)\tau_1\tau_3\big\mid\nonumber\\&\le&\frac{1}{3\pi^2}\mid\tau_1\left(1-\tau_1^2\right)\left(\mid\left\mid\frac{\tau_1^3}{6(1-\tau_1^2)}-\frac{(3+\tau_1^2)}{4\tau_1}\tau_2^2\right\mid+1-\mid\tau_2\mid^2\right)\nonumber\\&\le& \frac{1}{3\pi^2}\mid\tau_1\left(1-\tau_1^2\right)\mid\Psi\left(A_2, B_2, C_2\right), 
	 	\eea
	 	where $\Psi\left(A_2, B_2, C_2\right)=\mid A_2+B_2\tau_2+C_2\tau_2^2\mid+1-\mid\tau_2\mid^2$ with \beas A_2=\frac{\tau_1^3}{6\left(1-\tau_1^2\right)}, \  B_2=0 \ \text{and} \ C_2=-\frac{3+\tau_1^2}{4\tau_1}.\eeas
	  Since $\tau_1\in (0, 1)$. Clearly, we observed that $A_2C_2\le 0$. \\ Here by similar calculation done in Theorem {\ref{t3.2}} and by Lemma \ref{l2.5} we obtain:
	 	%\bea\label{s1} 
	 	%\left(\frac{1}{c_2^2}-1\right)=-\frac{1}{\left(3+\tau_1^2\right)^2(9-\tau_1^1)(1-\tau_1^2)}<0, \ \text{for} \ 0<\tau_1<1.
	 	%\eea
	 	%Therefore, it is obvious that $-4A_2C_2(C_2^{-2}-1)>0$, also $4(1+\mid C_2\mid)>0$ for all $0<\tau_1<1$. \\Thus we consider that
	 	%\beas B_2^2<\min\{-4A_2C_2(C_2^{-2}-1), 4(1+\mid C_2\mid)\}.\eeas
	 %	Hence, by Lemma {\ref{l2.5}} we get
	 %	Here we consider the following cases:\\ 
	 %\bea\label{s2} \Psi\left(A_2, B_2, C_2\right) =1+ \mid A_2\mid+\frac{B_2^2}{4(1+\mid C_2\mid)}.\eea
	% Using (\ref{m3.12}) and (\ref{s2}) 
	% we obtain
	 \bea\label{s3} \mid H_{2,1}(F_{f}/2)\mid&\le& \frac{1}{18\pi^2}\left(\tau_1^4-6\tau_1^3+6\tau_1\right)=\frac{1}{18\pi^2}P_2(\tau_1), \eea
	 where $P_2(\tau_1)=\tau_1^4-6\tau_1^3+6\tau_1$. 
	 Thus it is sufficient to determine the maximum of the auxiliary polynomial $$P_2(\tau_1)=\tau_1^4-6\tau_1^3+6\tau_1,\qquad 0\le\tau_1\le1.$$
	 The graph of $P_2$ is shown in Figure~\ref{fig:P2}.
	 \begin{figure}[t]
	 	\centering
	 	\begin{tikzpicture}
	 		\begin{axis}[
	 			width=12cm,
	 			height=7cm,
	 			xmin=0,
	 			xmax=1.05,
	 			ymin=0,
	 			ymax=3,
	 			domain=0:1,
	 			samples=300,
	 			axis lines=left,
	 			xlabel={$\tau_1$},
	 			ylabel={$P_2(\tau_1)$},
	 			xlabel style={font=\large},
	 			ylabel style={font=\large},
	 			clip=false,
	 			ymax=3.25,
	 			ymin=0,
	 			tick label style={font=\small},
	 			grid=major,
	 			major grid style={gray!25},
	 			enlargelimits=false,
	 			]

	 			\addplot[
	 			very thick,
	 			blue,
	 			]
	 			{x^4-6*x^3+6*x};

	 			\addplot[
	 			only marks,
	 			mark=*,
	 			mark size=2.8pt,
	 			red
	 			]
	 			coordinates {(0.62192,2.42675)};

	 		\draw[
	 		red,
	 		ultra thick,
	 		densely dashed,
	 		line cap=round
	 		] (axis cs:0.62192,0)
	 		--
	 		(axis cs:0.62192,2.42675);
	 		
	 		\draw[
	 		red,
	 		ultra thick,
	 		densely dashed,
	 		line cap=round
	 		] (axis cs:0,2.42675)
	 		--
	 		(axis cs:0.62192,2.42675);

	 			\node[
	 			anchor=west,
	 			font=\small
	 			]
	 			at (axis cs:0.54,2.67)
	 			{$(\tau_0,P_2(\tau_0))$};
	 			
	 		\end{axis}
	 	\end{tikzpicture}
	 	
	 	\caption{Graph of the auxiliary polynomial
	 		$P_2(\tau_1)=\tau_1^4-6\tau_1^3+6\tau_1$
	 		on $[0,1]$. The unique interior maximum occurs near
	 		$\tau_0\approx0.62192$.}
	 	\label{fig:P2}
	 \end{figure}
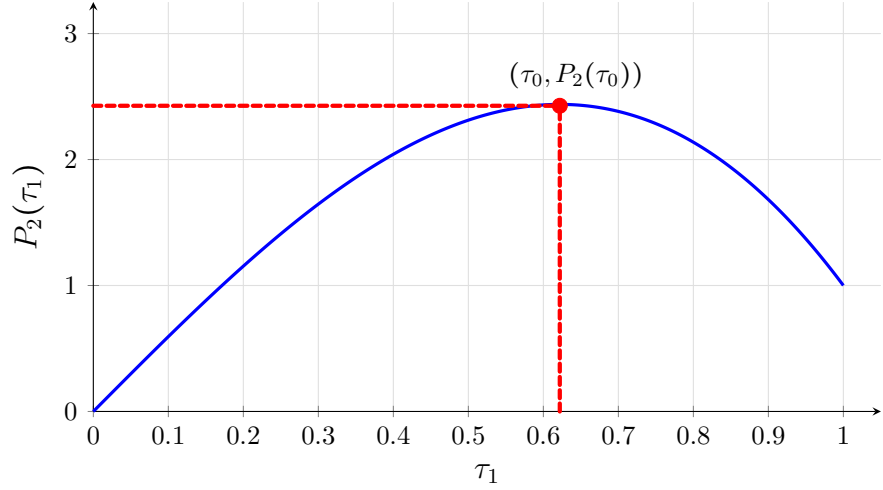
	 Differentiating,
	$$ P_2'(\tau_1) = 2(2\tau_1^3-9\tau_1^2+3).$$
	 The unique zero of $P_2'$ in $(0,1)$ is $\tau_0\approx0.62192$.
	 Moreover, $	 P_2''(\tau_0)<0$,
	 so $P_2$ attains its unique interior maximum at $\tau_0$.
	 	
	 	\smallskip
	 	Now, from (\ref{s3}) we get
	 	\beas 
	 	\mid H_{2,1}(F_{f}/2)\mid&\le& \frac{1}{18\pi^2}P_2(\tau_0)\approx 0.01365.
	 	\eeas
	 	Comparing all the estimate values obtained above, we conclude that $$|H_{2,1}(F_f/2)|	\le	\frac1{4\pi^2}.$$\\
	 	To see the inequality is sharp we choose the function $f_2$ defined by (\ref{f2}).
	 	\end{proof}
	 \begin{theo}
	 	Let $f(z)=z+a_2z^2+a_3z^3+\cdots\in\mathcal{S}^{*}_{\arcsin}$ and let $\gamma_1,\gamma_2$ be given by \eqref{e1.3}. Then
	 	\beas
	 	-\frac{1}{\pi} \ \le \ |\gamma_2|-|\gamma_1| \ \le \ \frac{1}{2\pi}.\eeas
	 	Both inequalities are sharp.
	 \end{theo}
	 \begin{proof}
	 	In view of (\ref{e1.3}) and (\ref{e3.3})--(\ref{e3.4}) we see that
	 	\bea\label{d3.14} \mid \gamma_2\mid -\mid\gamma_1\mid&=&\bigg\mid \frac{a_3}{2}-\frac{a_2^2}{4}\bigg\mid-\bigg\mid\frac{a_2}{2}\bigg\mid\nonumber\\&=&\frac{1}{2\pi}\left(\bigg\mid-\frac{c_1^2}{4}+\frac{c_2}{2}\bigg\mid-\bigg\mid c_1\bigg\mid\right)\nonumber\\&=&
	 	\frac{1}{2\pi}\left(\left|Kc_1^{2}+Lc_2\right|-J|c_1|\right)=\frac{1}{2\pi}\Phi\left(c_1, c_2\right),\eea 
	 	where $J=1$, $K=-\frac{1}{4}$ and $L=\frac{1}{2}$. Since $$|2K+L|=0<|L|+J=\frac32,$$
	 	the hypotheses of Lemma~\ref{l2.6} are satisfied.
	 	Therefore, by Lemma \ref{l2.6} we obtain
	 	\bea\label{d3.15}\mid \gamma_2\mid -\mid\gamma_1\mid =\frac{1}{2\pi}\Phi(c_1, c_2)\le \frac{1}{2\pi}.\eea
	 	To see the upper bound inequality is sharp we choose the function $f_2$ defined by (\ref{f2}).
	 	
	 	\smallskip
	 	For lower bound we see that $M=\mid 4K+2L\mid=0$. Moreover, $$M=|4K+2L|=0,$$
	 	and $$	J^2=1\le 2|L|(M+2|L|)=1.$$
	 	Hence Lemma~\ref{l2.6} yields 
	 	\bea\label{d3.16} \mid \gamma_2\mid-\mid\gamma_1\mid=\frac{1}{2\pi}\Phi(c_1, c_2)\ge -\frac{1}{\pi}.\eea
      Therefore, from (\ref{d3.15}) and (\ref{d3.16}) we obtain the desired bound.\\ 	For sharpness of the lower bound we consider the function $f_1$ introduced by (\ref{f1}). 
	 \end{proof}
	 \section{\bf{Sharp bound of logarithmic coefficient for inverse functions in the class $\mathcal{S}^{*}_{\arcsin}$}}
	  \begin{theo}
	 	Let $f(z)=z+a_2z^2+a_3z^3+\cdots\in\mathcal{S}^{*}_{\arcsin} $ and $\Gamma_1$, $\Gamma_2$, $\Gamma_3$ are given in \emph{(\ref{eil.4})}. Then
	 	$$
	 	|\Gamma_1| \ \le \ \frac{1}{\pi},
	 	\qquad |\Gamma_2| \ \le \ \frac{1}{2\pi},\qquad |\Gamma_3| \ \le \ \frac{1}{18\pi}+\frac{6}{\pi^3}.
	 	$$
	 	The estimates are sharp.  \end{theo}
	 	\begin{proof}
	 Using the coefficient relations (\ref{e3.3})--(\ref{e3.5}), obtained in the proof of Theorem~\ref{t3.1}, together with the inverse coefficient identities (\ref{eil.4}), we estimate each logarithmic coefficient separately.\\
	 		{\bf{A.}} From (\ref{eil.4}) and (\ref{e3.3}), we get 
	 		\beas \mid\Gamma_1\mid =\frac{1}{2\pi}\mid c_1\mid \ \le \ \frac{1}{\pi}.\eeas
	 		Hence, $|\Gamma_1|\le\frac1\pi$,
	 		which proves the first estimate.\par
	 		For sharpness we choose the extremal function $f_1$ defined by (\ref{f1}). \\
	 		{\bf{B.}} Using (\ref{eil.4}) and (\ref{e3.3})--(\ref{e3.4}) we obtain
	 		\beas \mid \Gamma_2 \mid &=& \frac{1}{2}\mid a_3-\frac{3}{2}a_2^2\mid\\&=&\frac{1}{4\pi}\mid c_2-\left(\frac{4+\pi}{2\pi}\right)c_1^2\mid.\eeas
	 		Applying Lemma~\ref{l2.1} with $ 		\mu=\frac{4+\pi}{2\pi}$, we obtain
	 		\beas \mid \Gamma_2 \mid\le \frac{1}{2\pi}.\eeas
	 		For sharpness we employ the function $f_2$ introduced by (\ref{f2}).\\
	 		{\bf{C.}} Substituting (\ref{e3.3})--(\ref{e3.5}) into the identity for $\Gamma_3$ in (\ref{eil.4}) yields
	 		\bea\label{e4.1} 
	 		\mid\Gamma_3\mid&=&\frac{1}{2}\left\mid a_4-4a_2a_3+\frac{10}{3}a_2^3\right\mid\nonumber\\ &=&
	 		\frac{1}{6\pi}\left\mid c_3-\left(\frac{2\pi+9}{2\pi}\right)c_1c_2+\left(\frac{7\pi^2+54\pi+108}{24\pi^2}\right)c_1^3\right\mid.
	 		\eea
	 		Now substituting value of $c_1$, $c_2$ and $c_3$ from (\ref{e2.2})--(\ref{e2.4}) into (\ref{e4.1}), then we have
	 		\bea\label{e4.2} \mid\Gamma_3\mid& =&\frac{1}{6\pi}\bigg\mid \left(\frac{\pi^2+108}{3\pi^2}\right)\tau_1^3-\frac{18}{\pi}\left(1-\tau_1^2\right)\tau_1\tau_2-2\left(1-\tau_1^2\right)\tau_1\tau_2^2\nonumber\\&&+2\left(1-\tau_1^2\right)\left(1-\mid\tau_2\mid^2\right)\tau_3\bigg\mid\eea
	 		Here, by Lemma \ref{l2.4}, we have $\tau_1\in [0, 1]$ and $\mid\tau_2\mid\le1$, $\mid\tau_3\mid\le 1$. Since $\tau_1\in[0,1]$, we distinguish the boundary cases $\tau_1=0$, $\tau_1=1$ and the interior case $0<\tau_1<1$.\\
	 		{\bf{Case--1.}} Let us choose $\tau_1=0$. Equation (\ref{e4.2}) reduces to \beas \mid \Gamma_3\mid= \frac{1}{3\pi}\left(1-\mid\tau_2\mid^2\right)\le \frac{1}{3\pi}\approx 0.10610.\eeas
	 		{\bf{Case--2.}} Let us choose $\tau_1=1$. Equation (\ref{e4.2}) simplifies to
	 		\beas \mid \Gamma_3\mid=\frac{\pi^2+108}{18\pi^3}\approx 0.21119.\eeas
	 		{\bf{Case--3.}} Let us choose $0<\tau_1<1$. Applying the triangle inequality to (\ref{e4.2}) together with $|\tau_3|\le1$, we arrive at
	 		\bea\label{e4.3} \mid\Gamma_3\mid &\le&\frac{1}{6\pi}\bigg(\bigg\mid\left(\frac{\pi^2+108}{3\pi^2}\right)\tau_1^3-\frac{18}{\pi}\left(1-\tau_1^2\right)\tau_1\tau_2-2\left(1-\tau_1^2\right)\tau_1\tau_2^2\bigg\mid+\bigg\mid2\left(1-\tau_1^2\right)\left(1-\mid\tau_2\mid^2\right)\tau_3\bigg\mid\bigg)\nonumber\\&\le&\frac{1}{3\pi}\big\mid1-\tau_1^2\big\mid\left(\left\mid\frac{\left(\pi^2+108\right)\tau_1^3}{6\pi^2(1-\tau_1^2)}-\frac{9}{\pi}\tau_1\tau_2-\tau_1\tau_2^2\right\mid+\big\mid\left(1-\mid\tau_2\mid^2\right)\big\mid\right)\nonumber\\&\le&\frac{1}{3\pi}\left(1-\tau_1^2\right)\Psi(A_3, B_3, C_3),\eea
	 		where, for convenience, $\Psi(A_3, B_3, C_3)=\mid A_3+B_3\tau_2+C_3\tau_2^2\mid+(1-\mid\tau_2\mid^2)$  with the values of \bea\label{e4.4} A_3=\frac{\left(\pi^2+108\right)\tau_1^3}{6\pi^2\left(1-\tau_1^2\right)}, \  B_3=-\frac{9\tau_1}{\pi} \ \text{and} \ C_3=-{\tau_1}.\eea
	 		In view of Lemma~\ref{l2.5}, we now distinguish the following cases according to the values of $A_3$, $B_3$ and $C_3$.
	 		
	 		\smallskip
	 		Suppose  $\tau_1\in (0, 1)$, then it is clear that $A_3C_3<0$.
	 		
	 		Since $A_3C_3<0$, Case (ii) of Lemma~\ref{l2.5} applies. Accordingly, we distinguish the following subcases.\\
	 		{\bf\underline{Subcase--3.1.}} Let us consider $0<\tau_1<\frac{2\pi}{2\pi+9}$ and let $-4A_3C_3\left(C_3^{-2}-1\right)<B_3^2$ and $\mid B_3\mid < 2(1-\mid C_3\mid)$.
	 		Now, from (\ref{e4.4}) we have
	 		\beas B_3^2+4A_3C_3\left(C_3^{-2}-1\right)=\frac{\left(2\pi^2+351\right)\tau_1^2}{3\pi^2}>0, \ \text{for} \ \tau_1\in(0, 1). \eeas
	 		Again, for $0<\tau_1<\frac{2\pi}{2\pi+9}\approx 0.41112$, we have $\mid B_3\mid<2\left(1-\mid C_3\mid\right)$. Therefore, by Lemma \ref{l2.5} we get
	 		\beas \Psi(A_3, B_3, C_3)&=&1-\mid A_3\mid+\frac{B_3^2}{4\left(1-\mid C_3\mid\right)}\\&=&\frac{1}{12\pi^2(1-\tau_1^2)}\left((27-2\pi^2)\tau_1^3+(243-12\pi^2)\tau_1^2+12\pi^2\right).\eeas
	 		Therefore, from (\ref{e4.3}) we get 
	 		\beas\mid\Gamma_3\mid &\le&\frac{1}{36\pi^3}\left((27-2\pi^2)\tau_1^3+(243-12\pi^2)\tau_1^2+12\pi^2\right)=\frac{1}{36\pi^3}P_3(\tau_1),  \eeas where $P_3(\tau_1)=(27-2\pi^2)\tau_1^3+(243-12\pi^2)\tau_1^2+12\pi^2$.
	 		Since, $0<\tau_1<\frac{2\pi}{2\pi+9}$, then it is easy to verify $P_3(\tau_1)$ is an increasing function.  So, the upper bound  
	 		\beas \mid\Gamma_3\mid \ \le 0.12542 \ (\text{approx}).\eeas\\
	 		{\bf\underline{Subcase--3.2.}} Let us consider the subcase $B_3^2<\min\left\{4(1+|C_3|)^2, -4A_3C_3(C_3^{-2}-1)\right\}$, which in particular requires $B_3^2<-4A_3C_3(C_3^{-2}-1)$.
	 		Using (\ref{e4.4}) it follows that \beas B_3^2\not<-4A_3C_3(C_3^{-2}-1) \ \text{for} \tau_1\in (0, 1).\eeas.
	 		Hence this alternative in Lemma~\ref{l2.5} cannot occur. So 
	 		\beas \Psi(A_3, B_3, C_3)\not=1+|A_3|+\frac{B_3^2}{4(1+|C_3|)}.\eeas\\
	 		{\bf\underline{Subcase--3.3.}} Let us consider the cases $\mid A_3B_3\mid\ge\mid C_3\mid \left(\mid B_3\mid +4\mid A_3\mid\right)$. Now, from (\ref{e4.4}) it follows that \beas 
	 		\mid A_3B_3\mid-\mid C_3\mid \left(\mid B_3\mid +4\mid A_3\mid\right)&=& \frac{\tau_1^2}{6\pi^3(1-\tau_1^2)}\left((-4\pi^3+63\pi^2-432\pi+972)\tau_1^2-54\pi^2\right)\\&=&\frac{\tau_1^2}{6\pi^3(1-\tau_1^2)}P_4(\tau_1),
	 		\eeas
	 		where $P_4(\tau_1)=(-4\pi^3+63\pi^2-432\pi+972)\tau_1^2-54\pi^2$. Since $0<\tau_1<1$, then $P_4(\tau_1)<0$. 
	 		Therefore this branch of Lemma~\ref{l2.5} is excluded and we have
	 		\beas \Psi(A_3, B_3, C_3)\neq|A_3|+\mid B_3\mid-|C_3|.\eeas \\ 
	 		{\bf\underline{Subcase--3.4.}} Let us consider the case $\mid A_3B_3\mid\le\mid C_3\mid \left(\mid B_3\mid -4\mid A_3\mid\right)$.
	 		Now, from (\ref{e4.4}) it follows that \beas 
	 		\mid A_3B_3\mid-\mid C_3\mid \left(\mid B_3\mid -4\mid A_3\mid\right)&=& \frac{\tau_1^2}{6\pi^3(1-\tau_1^2)}\left((4\pi^3+63\pi^2+432\pi+972)\tau_1^2-54\pi^2\right)\\&=&\frac{\tau_1^2}{6\pi^3(1-\tau_1^2)}P_5(\tau_1),
	 		\eeas
	 		where $P_5(\tau_1)=(4\pi^3+63\pi^2+432\pi+972)\tau_1^2-54\pi^2$. It follows by direct computation for $\tau_1\le\sqrt{\frac{54\pi^2}{4\pi^3+63\pi^2+432\pi+972}}\approx 0.41632$, we have $P_5(\tau_1)\le 0$. 
	 		Thus it sufficient to consider the interval $\frac{2\pi}{2\pi+9}\le \tau_1\le\sqrt{\frac{54\pi^2}{4\pi^3+63\pi^2+432\pi+972}}$. Therefore Lemma~\ref{l2.5} yields
	 		\beas \Psi(A_3, B_3, C_3)=-|A_3|+\mid B_3\mid+|C_3|.\eeas 
	 		Since, $\frac{2\pi}{2\pi+9}\le \tau_1\le\sqrt{\frac{54\pi^2}{4\pi^3+63\pi^2+432\pi+972}}$, then from (\ref{e4.3}) we get 
	 		\beas\mid\Gamma_3\mid &\le&\frac{1}{3\pi}\left(1-\tau_1^2\right)\Psi(A_3, B_3, C_3)\\&\le&\frac{1}{18\pi^3}\left(-(7\pi^2+54\pi+108)\tau_1^3+(6\pi^2+54\pi)\tau_1\right)=\frac{1}{18\pi^3}P_6(\tau_1),\eeas
	 		where $P_6(\tau_1)=-(7\pi^2+54\pi+108)\tau_1^3+(6\pi^2+54\pi)\tau_1$ is increasing function in the interval $\frac{2\pi}{2\pi+9}\le \tau_1\le\sqrt{\frac{54\pi^2}{4\pi^3+63\pi^2+432\pi+972}}$. 
	 		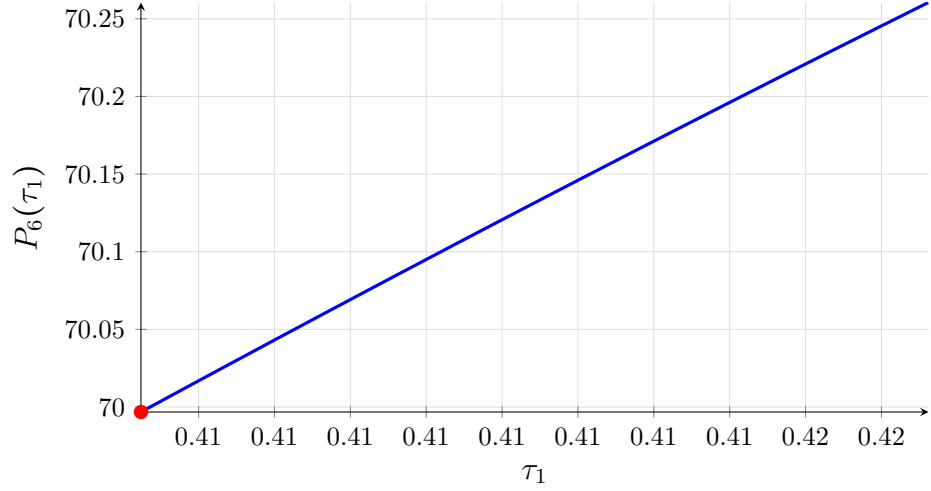
\begin{figure}[t]
	 			\centering
	 			\begin{tikzpicture}
	 				
	 				\pgfmathsetmacro{\a}{2*pi/(2*pi+9)}
	 				\pgfmathsetmacro{\b}{sqrt((54*pi*pi)/(4*pi*pi*pi+63*pi*pi+432*pi+972))}
	 				
	 				\begin{axis}[
	 					width=12cm,
	 					height=7cm,
	 					xmin=\a,
	 					xmax=\b,
	 					domain=\a:\b,
	 					samples=300,
	 					axis lines=left,
	 					xlabel={$\tau_1$},
	 					ylabel={$P_6(\tau_1)$},
	 					xlabel style={font=\large},
	 					ylabel style={font=\large},
	 					tick label style={font=\small},
	 					grid=major,
	 					major grid style={gray!25},
	 					enlargelimits=false,
	 					]

	 					\addplot[
	 					very thick,
	 					blue,
	 					]
	 					{-(7*pi*pi+54*pi+108)*x^3
	 						+(6*pi*pi+54*pi)*x};

	 					\addplot[
	 					only marks,
	 					mark=*,
	 					red,
	 					mark size=2.5pt
	 					]
	 					coordinates
	 					{
	 						(\a,{-(7*pi*pi+54*pi+108)*(\a)^3
	 							+(6*pi*pi+54*pi)*(\a)})
	 					};

	 					\addplot[
	 					only marks,
	 					mark=*,
	 					red,
	 					mark size=2.5pt
	 					]
	 					coordinates
	 					{
	 						(\b,{-(7*pi*pi+54*pi+108)*(\b)^3
	 							+(6*pi*pi+54*pi)*(\b)})
	 					};

	 					\draw[->,thick]
	 					(axis cs:0.413,48)
	 					--
	 					(axis cs:0.416,49);
	 					
	 					\node[font=\small]
	 					at (axis cs:0.4145,47.7)
	 					{Increasing};
	 					
	 				\end{axis}
	 			\end{tikzpicture}
	 			
	 			\caption{Graph of the auxiliary polynomial
	 				\[
	 				P_6(\tau_1)=-(7\pi^2+54\pi+108)\tau_1^3
	 				+(6\pi^2+54\pi)\tau_1
	 				\]
	 				on the interval
	 				\[
	 				\frac{2\pi}{2\pi+9}\le\tau_1\le
	 				\sqrt{\frac{54\pi^2}
	 					{4\pi^3+63\pi^2+432\pi+972}}.
	 				\]
	 				}
	 			\label{fig:P5}
	 		\end{figure}

	 		Therefore, the required bound 
	 		\beas\mid\Gamma_3\mid\le \frac{1}{18\pi^3}P_6(0.41632)\approx 0.12589. \eeas\\ 
	 		{\bf\underline{Subcase--3.5.}} Let us consider $\sqrt{\frac{54\pi^2}{4\pi^3+63\pi^2+432\pi+972}}<\tau_1<1$. It remains to consider the last case in Lemma \ref{l2.5}.
	 		\beas \Psi(A_3, B_3, C_3)=(|A_3|+|C_3|)\sqrt{1-\frac{B_3^2}{4A_3C_3}}.\eeas
	 		Here, from (\ref{e4.3}) and (\ref{e4.4}) we get
	 		\beas\mid \Gamma_3\mid\le \frac{\left((108-5\pi^2)\tau_1^2+6\pi^2\right)}{18\pi^3\sqrt{\left(2\pi^2+216\right)}}\sqrt{\left(243-(27-2\pi^2)\tau_1^2\right)}=\frac{1}{18\pi^3\sqrt{\left(2\pi^2+216\right)}}\beta(\tau_1),\eeas where $\beta(\tau_1)=\left((108-5\pi^2)\tau_1^2+6\pi^2\right)\sqrt{\left(243-(27-2\pi^2)\tau_1^2\right)}$.
	 	Here, it is easy to see that $\beta^{\prime}(\tau_1)>0$, then $\beta(\tau_1)$ is an increasing function. Thus the desired upper bound 
	 		\beas\mid\Gamma_3\mid \le \frac{1}{18\pi^3\sqrt{\left(2\pi^2+216\right)}}\beta(1)=\frac{1}{18\pi}+\frac{6}{\pi^3}\approx 0.21119. \eeas
	 		Combining Cases~1--3 together with Subcases~3.1--3.5, we conclude that
	 		$|\Gamma_3|\le\frac1{18\pi}+\frac6{\pi^3}$.\par 		The estimate is sharp, since equality is attained by the extremal function
	 		$f_1$ defined in (\ref{f1}).
	 	\end{proof}
	 	\begin{theo}
	 		Let $f(z)=z+a_2z^2+a_3z^3+\cdots\in\mathcal{S}^{*}_{\arcsin} $. Then Hankel determinant
	 		\beas\mid H_{2, 1}(F_{f^{-1}}/2)\mid \ \le \ \left(\frac{1}{18\pi^2}+\frac{2}{\pi^4}\right). \eeas
	 		The inequality is sharp.
	 	\end{theo}
	 	\begin{proof}
	 	Since $f\in S^{*}_{\arcsin}$, we know that 
	 	\bea\label{h4.7} H_{2, 1}(F_{f^{-1}}/2)=\Gamma_1\Gamma_3-\Gamma_2^2.\eea
	 	Using the identities \eqref{eil.4} and
	 	\eqref{e3.3}--\eqref{e3.5} in \eqref{h4.7},
	 	we obtain:
	 	\bea\label{h4.8} H_{2,1}(F_{f^{-1}}/2)&=&\frac{1}{4}\left(a_2a_4-a_2^2a_3-a_3^2+\frac{13}{12}a_2^4\right)\nonumber\\ &=&\frac{1}{12\pi^2}\left(c_1c_3-\frac{3}{4}c_2^2-\left(\frac{\pi+6}{4\pi}\right)c_1^2c_2+\left(\frac{5\pi^2+36\pi+72}{48\pi^2}\right)c_1^4\right).\eea
	 	Now substituting (\ref{e2.2})--(\ref{e2.4}) into (\ref{h4.8}) and then simplifying, we get 
	 	\bea\label{h4.9} 
	 	\mid H_{2,1}(F_{f^{-1}}/2)\mid &&=\frac{1}{12\pi^2}\bigg\mid\left(\frac{2\pi^2+72}{3\pi^2}\right)\tau_1^4-\frac{12}{\pi}\left(1-\tau_1^2\right)\tau_1^2\tau_2-\left(1-\tau_1^2\right)\left(3+\tau_1^2\right)\tau_2^2\nonumber\\&&+4\left(1-\tau_1^2\right)\left(1-\mid\tau_2\mid^2\right)\tau_1\tau_3\bigg\mid.
	 	\eea	
	 		By Lemma~\ref{l2.4}, $$	\tau_1\in[0,1],\qquad
	 		|\tau_2|\le1,\qquad	|\tau_3|\le1.$$
	 		We first examine the boundary cases $\tau_1=0$ and $\tau_1=1$ and then consider the interior interval $0<\tau_1<1$. \\
	 		{\bf\underline{Case--1.}} Let us choose $\tau_1=0$.
	 		 Then from (\ref{h4.9}), we get \beas\mid H_{2,1}(F_{f^{-1}}/2)\mid \ \le \ \frac{1}{4\pi^2}\mid\tau_2\mid^2\le \frac{1}{4\pi^2}\approx 0.02533. \eeas \\
	 		 {\bf\underline{Case--2.}} Let us choose $\tau_1=1$.
	 		 Then from (\ref{h4.9}), we get \beas\mid H_{2,1}(F_{f^{-1}}/2)\mid \ \le \ \left(\frac{1}{18\pi^2}+\frac{2}{\pi^4}\right)\approx 0.026161. \eeas\\
	 		 {\bf\underline{Case--3.}} Let us choose $0<\tau_1<1$. Then  applying the triangle inequality to (\ref{h4.9}), for $\tau_1\in (0, 1)$ and $\mid \tau_3\mid\le1$, we obtain the following inequality
	 	\bea\label{h4.10} \mid H_{2,1}(F_{f^{-1}}/2)\mid&\le&\frac{1}{3\pi^2}\tau_1\left(1-\tau_1^2\right)\bigg(\left\mid\frac{(\pi^2+36)\tau_1^3}{6\pi^2(1-\tau_1^2)}-\frac{3}{\pi}\tau_1\tau_2-\left(\frac{3+\tau_1^2}{4\tau_1}\right)\tau_2^2\right\mid+ \big\mid\left(1-\mid\tau_2\mid^2\right)\big\mid\bigg)\nonumber\\&\le&\frac{1}{3\pi^2}\tau_1\left(1-\tau_1^2\right)\Psi(A_4, B_4, C_4),\eea
	 	where $\Psi(A_4, B_4, C_4)=\mid A_4+B_4\tau_2+C_4\tau_2^2\mid+(1-\mid\tau_2\mid^2)$  with the values of \bea\label{h4.11} A_4=\frac{(\pi^2+36)\tau_1^3}{6\pi^2(1-\tau_1^2)}, \  B_4=-\frac{3}{\pi}\tau_1 \ \text{and} \ C_4=-\frac{3+\tau_1^2}{4\tau_1}.\eea
	 	In view of Lemma~\ref{l2.5}, we now distinguish the following cases according to the values of $A_4$, $B_4$ and $C_4$.
	 	
	 	\smallskip
	 	For $0<\tau_1<1$, it follows from
	 	\eqref{h4.11} that $A_4C_4<0.$
	 	Hence we apply Case (ii) of Lemma~\ref{l2.5}.
	 	So we consider following subcases:\\
	 	{\bf\underline{Subcase--3.1.}} Let us consider $-4A_4C_4\left(C_4^{-2}-1\right)<B_4^2$ and $\mid B_4\mid < 2(1-\mid C_4\mid)$.
	 	Now, from (\ref{h4.11}) we deduce that
	 	\beas \mid B_4\mid-2\left(1-\mid C_4\mid\right)=\frac{1}{2\pi\tau_1}\left((\pi+6)\tau_1^2-4\pi\tau_1+3\pi\right)>0, \ \text{for} \ \tau_1\in(0, 1). \eeas
	 	Therefore, by Lemma \ref{l2.5} we get
	 	\beas \Psi(A_4, B_4, C_4)\neq1-\mid A_4\mid+\frac{B_4^2}{4\left(1-\mid C_4\mid\right)}.  \eeas\\
	 	{\bf\underline{Subcase--3.2.}} Let us consider the case $B_4^2<\min\left\{4(1+|C_4|)^2, -4A_4C_4(C_4^{-2}-1)\right\}$. Now from (\ref{h4.11}) we deduce that $\left({C_4^{-2}}-1\right)<0$.
	 	As, we have $A_4C_4<0$, then $-4A_4C_4\left({C_4^{-2}}-1\right)<0$ for all $0<\tau_1<1$. Also, it follows that 
	 	\beas\min\left\{4(1+|C_4|)^2, -4A_4C_4(C_4^{-2}-1)\right\}=-4A_4C_4\left({c_4^{-2}}-1\right).\eeas
	 	Thus, it is easy to verify that $B_4^2\not<\min\left\{4(1+|C_4|)^2, -4A_4C_4(C_4^{-2}-1)\right\}$.\\
	 	Therefore, by Lemma \ref{l2.5} it is clear that 
	 	$ \Psi(A_4, B_4, C_4)\neq1+\mid A_4\mid+\frac{B_4^2}{4\left(1+\mid C_4\mid\right)}$.\\
	 {\bf\underline{Subcase--3.3.}} Let us consider the cases $\mid A_4B_4\mid\ge\mid C_4\mid \left(\mid B_4\mid +4\mid A_4\mid\right)$. Since $0<\tau_1<1$, then from (\ref{h4.11}) we have
	 	\beas &&\mid A_4B_4\mid-\mid C_4\mid \left(\mid B_4\mid +4\mid A_4\mid\right)\\&&=-\frac{1}{12\pi^3\left(1-\tau_1^2\right)}\big( (2\pi^3-15\pi^2+72\pi-216)\tau_1^4+6\pi(\pi^2-3\pi+36)\tau_1^2+27\pi^2\big)<0.\eeas
	 	Hence, using Lemma \ref{l2.5} we get
	 	$\Psi(A_4, B_4, C_4)\neq\mid A_4\mid+\mid B_4\mid-\mid C_4\mid$.\\
	 	{\bf\underline{Subcase--3.4.}} Let us consider the cases $\mid A_4B_4\mid\le\mid C_4\mid \left(\mid B_4\mid -4\mid A_4\mid\right)$. So, from (\ref{h4.11}) we have
	 	\beas &&\mid A_4B_4\mid-\mid C_4\mid \left(\mid B_4\mid -4\mid A_4\mid\right)\\&&=\frac{1}{12\pi^3\left(1-\tau_1^2\right)}\big( (2\pi^3+15\pi^2+72\pi+216)\tau_1^4+6\pi(\pi^2+3\pi+36)\tau_1^2-27\pi^2\big).\eeas
	 	Let us choose $t=\tau_1^2$ and $P_7(t)=(2\pi^3+15\pi^2+72\pi+216)t^2+6\pi(\pi^2+3\pi+36)t-27\pi^2$.
	 Thus it follows by direct calculation for $0<t\le 0.22421$(approx), we have $P_7(t)\le 0$. Then this implies that for $0<\tau_1\le 0.47351$ we get $\mid A_4B_4\mid\le\mid C_4\mid \left(\mid B_4\mid -4\mid A_4\mid\right)$.\\ 
	 	Therefore, using Lemma \ref{l2.5} and from (\ref{h4.10}) we get
	 	\bea\label{h4.12} H_{2,1}(F_{f^{-1}}/2)&\le&\frac{1}{3\pi^2}\tau_1\left(1-\tau_1^2\right)\Psi(A_4, B_4, C_4)\nonumber\\&\le&\frac{1}{36\pi^4}\left(-(5\pi^2+36\pi+72)\tau_1^4+6\pi(6-\pi)\tau_1^2+9\pi^2\right)=\frac{1}{36\pi^4}P_8(\tau_1),  \eea
	 	where $P_8(\tau_1)=-(5\pi^2+36\pi+72)\tau_1^4+6\pi(6-\pi)\tau_1^2+9\pi^2$. The graph of $P_8$ is shown in Figure~\ref{fig:P7}.

	 	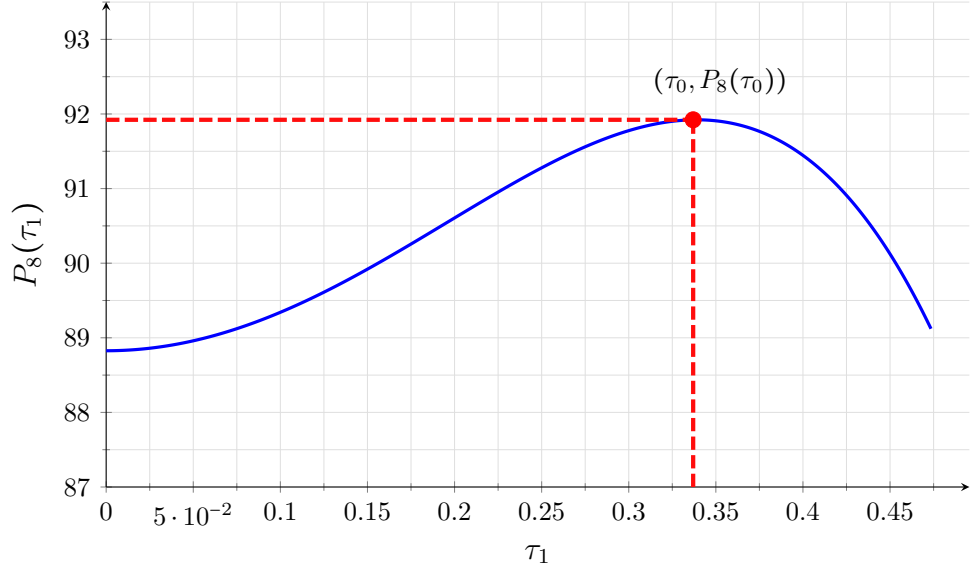
\begin{figure}[t]
	 		\centering
	 		
	 		\begin{tikzpicture}
	 			
	 			\begin{axis}[
	 				width=13cm,
	 				height=8cm,
	 				xmin=0,
	 				xmax=0.49551,
	 				ymin=87,
	 				ymax=93.5,
	 				domain=0:0.47351,
	 				samples=400,
	 				axis lines=left,
	 				xlabel={$\tau_1$},
	 				ylabel={$P_8(\tau_1)$},
	 				xlabel style={font=\large},
	 				ylabel style={font=\large},
	 				tick label style={font=\small},
	 				grid=major,
	 			major grid style={
	 				gray!25,
	 				line width=0.3pt
	 			},
	 			minor grid style={
	 				gray!25,
	 				line width=0.3pt
	 			},
	 			grid=both,
	 			minor tick num=1,
	 				]
	 				
	 				\addplot[
	 				blue,
	 				very thick
	 				]
	 				{-(5*pi^2+36*pi+72)*x^4
	 					+6*pi*(6-pi)*x^2
	 					+9*pi^2};
	 				
	 				\addplot[
	 				only marks,
	 				mark=*,
	 				mark size=3pt,
	 				red
	 				]
	 				coordinates
	 				{
	 					(0.33698,91.92207)
	 				};
	 				
	 				\addplot[
	 				red!95!,
	 				densely dashed,
	 				ultra thick,
	 				dash pattern=on 5pt off 2pt
	 				]
	 				coordinates
	 				{
	 					(0.33698,\pgfkeysvalueof{/pgfplots/ymin})
	 					(0.33698,91.92207)
	 				};

	 				\addplot[
	 				red!95!,
	 				densely dashed,
	 				ultra thick,
	 				dash pattern=on 5pt off 2pt
	 				]
	 				coordinates
	 				{
	 					(0,91.92207)
	 					(0.33698,91.92207)
	 				};

	 				\node[
	 				anchor=south east,
	 				font=\small
	 				]
	 				at (axis cs:0.39698,92.12207)
	 				{$(\tau_0,P_8(\tau_0))$};
	 				
	 			\end{axis}
	 			
	 		\end{tikzpicture}
	 		
	 		\caption{Graph of the auxiliary polynomial
	 			$P_8(\tau_1)=-(5\pi^2+36\pi+72)\tau_1^4
	 			+6\pi(6-\pi)\tau_1^2+9\pi^2$
	 			on the interval $0<\tau_1\le0.47351$. The graph shows the unique interior maximum attained near $\tau_0\approx0.33698$.}
	 		\label{fig:P7}
	 	\end{figure}

	 	It is clear that for $0<\tau_1\le 0.47351$ maximum value of $P_8(\tau_1)$ exists at $\tau_0\approx 0.33698$. Hence from (\ref{h4.12}) we get
	 	\beas H_{2,1}(F_{f^{-1}}/2)\le \frac{1}{36\pi^4}P(0.33698)\approx 0.025981.\eeas\\
	    {\bf\underline{Subcase--3.5.}} Let us consider $0.47351<\tau_1<1$. It remains to consider the last case in Lemma \ref{l2.5}. Then by (\ref{h4.10}) and (\ref{h4.11}) we have
	 	\bea\label{h4.13} 
	 	\mid H_{2, 1}(F_{f^{-1}}/2)\mid&\le&\frac{1}{3\pi^2}\tau_1(1-\tau_1^2)(|C_4|+|A_4| )\sqrt{1-\frac{B_4^2}{4A_4C_4}}\nonumber\\ &\le& \left(\frac{(72-\pi^2)\tau_1^4-6\pi^2\tau_1^2+9\pi^2}{36\pi^4\sqrt{\pi^2+36}}\right)\sqrt{\frac{3\pi^2+162+(\pi^2-18)\tau_1^2}{3+\tau_1^2}}\nonumber\\ &\le&\frac{1}{36\pi^4(\sqrt{\pi^2+36})}\eta(\tau_1),\eea
	 	where $\eta(\tau_1)={(72-\pi^2)\tau_1^4-6\pi^2\tau_1^2+9\pi^2}\sqrt{\frac{3\pi^2+162+(\pi^2-18)\tau_1^2}{3+\tau_1^2}}$, for $0.47351<\tau_1<1$. Now, differentiate this with respect to $\tau_1$, we get
	 	\beas\eta^{\prime}(\tau_1)=\frac{4\tau_1P_8(\tau_1)}{{\left(3+\tau_1^{2}\right)^2} \sqrt{\frac{3\pi^2+162+(\pi^2-18)\tau_1^2}{3+\tau_1^2}}},\eeas
	 	where $$P_8(\tau_1)=-(\pi^4-90\pi^2+1296)\tau_1^6+(-9\pi^4+432\pi^2+3888)\tau_1^4+(-27\pi^4+162\pi^2+34992)\tau_1^2-27\pi^4-1944\pi^2$$. Solving $P_8(\tau_1)=0$ numerically for $\tau_1\in (0, 1)$, we find a unique solution $\tau_0^{\prime}\approx0.75797$.
	 	
	 	\smallskip
	 	Since $\tau_0^{\prime}\in (0.47351, 1)$, then it is clear that $P_8(\tau_1)$ is decreasing in $(0.47351, \tau_0^{\prime}]$ and increasing in $[\tau_0^{\prime}, 1)$.
	 	
	 	\smallskip 
	 	Therefore, for $0.47351<\tau_1\le \tau_0^{\prime}$ and from (\ref{h4.13}) we get 
	 	\beas	\mid H_{2, 1}(F_{f^{-1}}/2)\mid\le \frac{\eta{(0.75797)}}{36\pi^2\sqrt{\pi^2+36}}\approx 0.02293. \eeas
	 	
	 	\smallskip
	 	Similarly for $\tau_0^{\prime}\le\tau_1<1$ and from (\ref{h4.13}) we get 
	 	\beas	\mid H_{2, 1}(F_{f^{-1}}/2)\mid\le \left(\frac{1}{18\pi^2}+\frac{2}{\pi^4}\right)\approx 0.026161. \eeas
	 	Comparing the estimates obtained in Cases~1--3, we conclude that $$|H_{2,1}(F_{f^{-1}}/2)|	\le
	 	\frac{1}{18\pi^2}+\frac{2}{\pi^4},$$
	 	which is the desired estimate.\\
	 	Equality is attained by the extremal function
	 	$f_1$ defined in \eqref{f1}. 
	 	
	 	\end{proof}
	 	\begin{theo}
	 		Let $f(z)=z+a_2z^2+a_3z^3+\cdots\in \mathcal{S}_{\arcsin}^{*}$ and let $\Gamma_1, \ \Gamma_2$ be given by \eqref{e1.3}. Then
	 		\beas
	 		-\sqrt{\frac{1}{\pi(4+\pi)}} \ \le \ |\Gamma_2|-|\Gamma_1| \ \le \ \frac{1}{2\pi}.\eeas
	 		Both inequalities are sharp.
	 	\end{theo}
	 	\begin{proof}
	 			In view of (\ref{e1.3}) and (\ref{e3.3})--(\ref{e3.4}) we see that
	 			\bea\label{i3.14} \mid \Gamma_2\mid -\mid\Gamma_1\mid&=&\bigg\mid \frac{a_3}{2}-\frac{3a_2^2}{4}\bigg\mid-\bigg\mid\frac{a_2}{2}\bigg\mid\nonumber\\&=&\frac{1}{2\pi}\left(\bigg\mid-\left(\frac{4+\pi}{4\pi}\right)c_1^2+\frac{c_2}{2}\bigg\mid-\bigg\mid c_1\bigg\mid\right)\nonumber\\&=&
	 			\frac{1}{2\pi}\left(\left|Kc_1^{2}+Lc_2\right|-J|c_1|\right)=\frac{1}{2\pi}\Phi\left(c_1, c_2\right),\eea 
	 			where $K=-\frac{4+\pi}{4\pi}$, $L=\frac{1}{2}$ and $J=1$.
	 			
	 			\smallskip
	 			 Now, for upper bound it is clear that $\mid 2K+L\mid \ \not\ge \ \mid L\mid+J$. Then by Lemma \ref{l2.6}  in (\ref{i3.14}) we obtain
	 			\bea\label{i3.15}\frac{1}{2\pi}\Phi(c_1, c_2)=\mid \Gamma_2\mid -\mid\Gamma_1\mid \le \frac{1}{2\pi}.\eea
	 			For sharpness of the upper bound we consider $f_2$ introduced by (\ref{f2}).
	 			
	 			\smallskip
	 			For lower bound we see that $M=\mid 4K+2L\mid=\frac{4}{\pi}$. Also it is clear that $J^2\le 2\mid L\mid\left(M+2\mid L\mid\right)$. Thus by Lemma \ref{l2.6} we have 
	 			\bea\label{i3.16} \mid\Gamma_2\mid-\mid\Gamma_1\mid=\frac{1}{2\pi}\Phi(c_1, c_2)\ge -\sqrt{\frac{1}{\pi(4+\pi)}}.\eea
	 			Therefore, from (\ref{i3.15}) and (\ref{i3.16}) we obtain the desired bound.\\
	 			For sharpness of the lower bound we consider the function 
	 			$ p(z)=\frac{1+2\sqrt{\frac{\pi}{\pi+4}}z+z^2}{1-z^2}$, then we have $\omega(z)=\frac{z\left(z+\sqrt{\frac{\pi}{\pi+4}}\right)}{\left(1+\sqrt{\frac{\pi}{\pi+4}}\right)}$. Using (\ref{e3.1}) we get 
	 			\beas f_5(z)=z\exp\left({\frac{2}{\pi}\int_{0}^{z}\frac{\arcsin \omega(t)}{t}dt}\right).\eeas
	 			Hence, proof of the theorem is complete.
	 		\end{proof}
	 		
	 		\vspace{5mm}
	 		
	 		\noindent\textbf{Conflict of interest:} The authors declare that there is no conflict  of interest regarding the publication of this paper.\vspace{1.2mm}
	 		
	 		\noindent {\bf Funding:} Not Applicable.\vspace{1.2mm}
	 		
	 		\noindent\textbf{Data availability statement:}  Data sharing not applicable to this article as no datasets were generated or analysed during the current study.\vspace{1.2mm}
	 		
	 		\noindent {\bf Authors' contributions:} All the authors have equal contributions in preparation of the manuscript.


\begin{thebibliography}{99}
	
	\bibitem{Ali_BMSS_2003} R. Ali, Coefficients of the inverse of strongly starlike functions, {\textit{Bull. Malays. Math. Sci. Soc.}}, {\bf{26}}(2003), 63--71.
	
	\bibitem{Ali_Vasudevarao_PAMS_2018} M. F. Ali and V. Allu, On logarithmic coefficients of some close-to-convex functions, {\textit{Proc. Amer. Math. Soc.}}, {\bf{146}}(3)(2018), 1131--1142.
	
	\bibitem{Ali_Allu_BAMS} M. F. Ali and V. Allu, Logarithmic coefficients of some close-to-convex functions, {\textit{Bull. Aust. Math. Soc.}}, {\bf{95}}(2017), 228-237.
	
	\bibitem{Alotabi et. al.} A. Alotaibi, M. Arif, M. A. Alghamdi and S. Hussain, Starlikeness associated with cosine hyperbolic function, {\textit{Math.}}, {\bf 8}(2020), 1118
	
    \bibitem{ASS} A. Banerjee, S. Majumder and S. Panja, Hankel and Toeplitz determinant estimates for starlike functions via Schwarz function techniques,{\textit{ Ukr. Math. J.}}, (Accepted).
	
	\bibitem{Bano Raza} K. Bano, M. Raza, Starlike functions associated with cosine functions, {\textit{Bull. Iran. Math. Soc.}}, {\bf{47}}(2021), 1513--1532.
	
    \bibitem{cho et al} N. E. Cho, V. Kumar, S. S. Kumar, V. Ravichandran, Radius problems for starlike functions associated with the sine function, {\textit{Bull. Iran. Math. Soc.}}, {\bf{45}}(2019), 213--232.
	
    \bibitem{CKL1} N. E. Cho, B. Kowalczyk and A. Lecko, Sharp bounds of some coefficient functionals over the class of functions convex in the direction of the imaginary axis, {\textit{Bull. Aust. Math. Soc.}}, {\bf{100}}(2019), 86--96.
	
	\bibitem{CKS1} J. H. Choi, Y. C. Kim and T. Sugawa, A general approach to the Fekete-Szeg\"{o} problem, {\textit{J. Math. Soc. Japan}}, {\bf{59}}(2007), 707--727.
	
	\bibitem{PL_D_1983} P. L. Duren, Univalent Functions, Springer-Verlag, New York, 1983.
	
	\bibitem{Janowski_Anpolon_1970} W. Janowski, Extremal problems for a family of functions with positive real part and for some related families, {\textit{Ann. Polon. Math.}}, {\bf{23}}(1970/71), 159--177.
	
    \bibitem{KL_BAMS_2022} B. Kowalczyk and A. Lecko, Second Hankel determinant of logarithmic coefficients of convex and starlike functions, {{Bull. Aus. Math. Soc.}}, {\bf{105}}(3)(2022), 458--467.
	
	\bibitem{KL_RACSAM_2023} B. Kowalczyk and A. Lecko, The second Hankel determinant of the logarithmic coefficients of strongly starlike and strongly convex functions, {\textit{Rev. R. Acad. Cienc. Exactas F´ıs. Nat. Ser. A Mat. RACSAM }}, {\bf{117}}(2)(2023), Paper No. 91.
	
	\bibitem{Kumar_AMP_2021} S. S. Kumar and G. Kamaljeet, A cardioid domain and starlike functions, {\textit{Anal. Math. Phys.}}, {\bf{11}}(2)(2021), 34 pp.
	
	\bibitem{Kumar et. al.} S. S. Kumar, M. G. Khan, B. Ahmad, and W. K. Mashwani, A class of analytic functions associated with sine hyperbolic functions, {\textit{J. Anal.}}, {\bf 32}(5)(2024), 3065--3085.
	
	\bibitem{Kumar_2024} S. Kumar, R. K. Pandey and P. Rai, Sharp bounds on Hankel and Hermitian–Toeplitz determinants of associated Sakaguchi functions, {\textit{Lith. Math. J.}}, {\bf{64}}(4)(2024), 491-506.
	
	\bibitem{Kumar_Yadav_IJS_2026} S. S. Kumar and P. Yadav, On a Class of Starlike Functions Associated with a Bean Shaped domain, {\textit{Iran. J. Sci.}}, {\bf{50}}(2026), 197--211.
	
	\bibitem{Lecko et al} A. Lecko, A. and B. Smiarowska, Sharp inequalities for Zalcman functional of logarithmic coefficients of inverse functions in certain classes of analytic functions, {\textit{J. Math. Inequal.}}, {\bf19}(1)(2025), 81--97. 

	\bibitem{Lowner_annalen_1923} K. L\"{o}wner, Untersuchungen \"{u}ber schlichte konforme Abbildungen des Einheitskreises. I, {\textit{Math. Annalen}}, {\bf{89}}(1)(1923), 103--121.
	
	\bibitem{Ma+Minda_1992} W. C. Ma and D. Minda, A unified treatment of some special classes of univalent functions,{\textit{Proceedings of the Conference on Complex Analysis}} (Tianjin, 1992), pp. 157–169, Conf.
	Proc. Lecture Notes Anal., Vol. I, International Press, Cambridge, MA, 1994.
	
	\bibitem{Majumder et al._JA_2026} S. Majumder, D. Pramanik and N. Sarkar, The second Hankel determinant for logarithmic coefficients of inverse convex functions of a given order, {J. Anal.}, {\bf{34}}(2026), 365--388.
	
	\bibitem{Majumder_Indian} S. Majumder, N. Sarkar and M. B. Ahamed, On coefficient problems for classes $\mathcal{S}^{*}_e$ and $\mathcal{C}_e$, {\textit{Indian J. Pure Appl. Math.,}} (2026). 
	
	\bibitem{Mwndiratta et al_BMMS_2015} R. Mendiratta, S. Nagpal, and V. Ravichandran, On a subclass of strongly starlike functions associated with exponentia function, {\textit{Bull. Malays. Math. Sci. Soc.}}, {\bf{38}}(1)(2015), 365--386.
	
	\bibitem{Milin_1971} I. M. Milin, Univalent Functions and Orthonormal Systems, {\textit{Transl. Math. Monogr.}}, Vol. {\bf{49}}, AMS, 1977.
	
	\bibitem{MK_BIMS_2023}  M. Mundalia and S. S. Kumar, Coefficient problems for certain close-to-convex functions, {\textit{Bull. Iran. Math. Soc.}} {\bf{49}}(1)(2023), Paper No. 5.
	
	\bibitem{Nagpal_2025} S. Nagpal, Hankel determinant for two subclasses of univalent functions with bounded turning, {\textit{Lith. Math. J.}}, {\bf{65}}(2)(2025), 294-305.
	
    \bibitem{PSW_RM} S. Ponnusamy, N. L. Sharma and K. J. Wirths, Logarithmic Coefficients of the Inverse of Univalent Functions, {\textit{Result. Math.}}, {\bf{73}}(2018), 160.
	
	\bibitem{Ponnusamy_BSM_2021} S. Ponnusamy and T. Sugawa, Sharp inequalities for logarithmic coefficients and their applications, {\textit{Bull. Sci. Math.}} {\bf{166}}(2021), 102931.
	
	\bibitem{Ponnusamy_JAMS_2020} S. Ponnusamy, N. L. Sharma, and K. J. Wirths, Logarithmic coefficients problems in families related to starlike and convex functions, {\textit{J. Aust. Math. Soc.}}, {\bf{109}}(2020), 230--249.
	
	\bibitem{Pom_1967} C. Pommerenke, On the Hankel determinants of univalent functions, {\textit{Math.}}, {\bf 14}(1)(1976), 108--112.
	
	\bibitem{Raina_HJMS_2015} R. k. Raina and J. Sok\'{o}l, On coefficient estimates for a certain class of starlike functions, {\textit{Hacet. J. Math. Stat.}}, {\bf{44}}(6)(2015), 1427--1433.
	
	\bibitem{Ravichandran_Verma_CRMAA_2015} V. Ravichandran and S. Verma, Bound for the fifth coefficient of certain starlike functions, {\textit{C. R. Math. Acad. Sci. Paris}}, {\bf{353}}(2015), 505--510.
	
	\bibitem{Raza et. al} M. Raza, A. Riaz, D. K. Thomas and P. Zaprawa, Third Hankel determinant for starlike and convex functions associated with the exponential function, {\textit{Bol. Soc. Mat. Mex.}}, {\bf{31}}(1)(2025), 16.
	
	\bibitem{Raza_Riaz} M. Raza, A. Riaz, and P. Zaprawa, Hankel determinants for starlike functions associated with a cardioid domain, {\textit{Lith. Math. J.}}, {\bf{65}}(3)(2025), 368-379.
	
	\bibitem{Ronning_AUMC_1991} F. Rønning, On starlike functions associated with parabolic regions, {\textit{Ann. Univ. Mariae Curie-Sk lodowska Sect.}},  {\bf{45}}(1991), 117--122.
	
	\bibitem{BBS_2024} B. Şeker, B. Çekiç, S. Sümer and O. Akçiçek, The second Hankel determinant of logarithmic coefficients and logarithmic inverse coefficients for the class of bounded turning functions of order $\alpha$, {\textit{Lith. Math. J.}}, {\bf 64}(4)(2024), 530-545.
	
	\bibitem{Sim_Thomas_Symmetry_2020} Y. J. Sim and D. K. Thomas, On the difference of inverse coefficients of univalent functions, {\textit{Symmetry}}, {\bf{12}}(12)(2020).
	
	\bibitem{Stainkiewicz} J. Sok\'{o}l and J. Stankiewicz, Radius of convexity of some subclasses of strongly starlike functions, {\textit{Zesz. Nauk. Politech. Rzesz. Mat.}}, {\bf{19}}(1996), 101--105.
	
	\bibitem{S_Arif_Lithunian} L. Shi and M. Arif, Sharp bounds on the third Hankel determinant for the Ozaki close-to-convex and convex functions, {\textit{Lith. Math. J.}}, {\bf{63}}(4)(2023), 487-504.
	
	\bibitem{Thomas_PAMS_2016} D. Thomas, On the logarithmic coefficients of close-to-convex functions, {\textit{Proc. Amer. Math. Soc.}}, {\bf{144}}(4)(2016), 1681--1687.
	
	\bibitem{Zaprawa_BSMM_2021} P. Zaprawa, Initial logarithmic coefficients for functions starlike with respect to symmetric points, {\textit{Bol. Soc. Mat. Mex.}}, {\bf{27}}(3)(2021), Paper No. 62, 13 pp.
	
	\bibitem{Zaprawa} P. Zaprawa, Hankel determinant $H_{2, 3}$ for starlike and convex functions, {\textit{Bull. Sci. Math.}}, {\bf 194}(2024), Art., ID 103459.
	
	
	
	
	
	
	
	
	
	
	
	

	

	

	

	
	
	
	
	
	
	
	
	
	
	
\end{thebibliography}
\end{document}